\documentclass[11pt]{amsart}
\usepackage[margin=1in]{geometry}
\usepackage{graphicx}
\usepackage{setspace}
\usepackage{amssymb,amsmath,amsthm,amsfonts,amscd,centernot}
\usepackage{hyperref}
\usepackage{color}
\usepackage{booktabs}
\usepackage{tabularx}
\usepackage{enumitem}
\usepackage[retainorgcmds]{IEEEtrantools}
\usepackage[notref,notcite,final]{showkeys}
\usepackage[final]{pdfpages}
\usepackage{fancyhdr}
\usepackage{ytableau}
\usepackage{stmaryrd}
\usepackage{cleveref}
\usepackage[T1]{fontenc}
\usepackage{verbatim}
\usepackage{parskip}
\usepackage{setspace}

\makeatletter
\newcommand{\fixed@sra}{$\vrule height 2\fontdimen22\textfont2 width 0pt\rightarrow$}
\newcommand{\diagarrow}{%
\mathrel{\text{\rotatebox[origin=c]{\numexpr45}{\fixed@sra}}}
}
\makeatother

\usepackage{tikz}
\usepackage{tikz-cd}
\usetikzlibrary{cd}

\newtheorem{thm}{Theorem}[section]
\newtheorem*{thm*}{Theorem}
\newtheorem{lemma}[thm]{Lemma}
\newtheorem*{lemma*}{Lemma}
\newtheorem{prop}[thm]{Proposition}

\newtheorem{corollary}[thm]{Corollary}
\newtheorem{conjecture}[thm]{Conjecture}
\newtheorem{question}[thm]{Question}

\newtheorem{Definition}[thm]{Definition}
\newenvironment{definition}
{\begin{Definition}\rm}{\end{Definition}}
\newtheorem{Example}[thm]{Example}
\newenvironment{example}
{\begin{Example}\rm}{\end{Example}}
\newtheorem{Algorithm}[thm]{Algorithm}

\theoremstyle{definition}
\newtheorem{remark}[thm]{\textbf{Remark}}

\newcommand{\N}{\mathbb{N}}

\newcommand{\T}{\mathcal{T}}
\newcommand{\Z}{\mathbb{Z}}

\newcommand{\Prob}{\mathbb{P}}

\newcommand{\call}{\mathcal{L}}

\DeclareMathOperator{\im}{\mathrm{im}}

\newcommand{\inverse}{^{-1}}

\newcommand{\st}{\mathrm{st}}
\newcommand{\vphi}{\varphi}

\setlist[enumerate,1]{leftmargin=1.8em, itemsep=0mm, label=\textup{(}\arabic*\textup{)}}

\begin{document}
\title{On Promotion and Quasi-Tangled Labelings of Posets}

\author{Eliot Hodges}
\address{Department of Mathematics, Harvard University, Cambridge, MA 02138}
\email{eliothodges@college.harvard.edu}
\begin{abstract}
	In 2022, Defant and Kravitz introduced extended promotion (denoted $ \partial $), a map that acts on the set of labelings of a poset. Extended promotion is a generalization of Sch\"{u}tzenberger's promotion operator, a well-studied map that permutes the set of linear extensions of a poset. It is known that if $ L $ is a labeling of an $ n $-element poset $ P $, then $ \partial^{n-1}(L) $ is a linear extension. This allows us to regard $ \partial $ as a sorting operator on the set of all labelings of $ P $, where we think of the linear extensions of $ P $ as the labelings which have been sorted. The labelings requiring $ n-1 $ applications of $ \partial $ to be sorted are called \emph{tangled}; the labelings requiring $ n-2 $ applications are called \emph{quasi-tangled}. In addition to computing the sizes of the fibers of promotion for rooted tree posets, we count the quasi-tangled labelings of a relatively large class of posets called \emph{inflated rooted trees with deflated leaves}. Given an $ n $-element poset with a unique minimal element with the property that the minimal element has exactly one parent, it follows from the aforementioned enumeration that this poset has $ 2(n-1)!-(n-2)! $ quasi-tangled labelings. Using similar methods, we outline an algorithmic approach to enumerating the labelings requiring $ n-k-1 $ applications to be sorted for any fixed $ k\in\{1,\ldots,n-2\} $. We also make partial progress towards proving a conjecture of Defant and Kravitz on the maximum possible number of tangled labelings of an $ n $-element poset.
\end{abstract}
\maketitle
\section{Introduction}
\subsection{Background}
Let $ P $ be an $ n $-element poset, whose order relation we denote by $ <_P $. A \emph{labeling} of $ P $ is a bijection $ L:P\to[n] $ (where $ [n]=\{1,\ldots,n\} $). A labeling $ L $ is called a \emph{linear extension} if it preserves the order on $ P $, i.e.\ if for all pairs $ x,y\in P $ with $ x<_Py $ we have $ L(x)<L(y) $. Let $ \Lambda(P) $ be the set of all labelings of $ P $; let $ \call(P)\subset\Lambda(P) $ be the subset consisting of all linear extensions.

In \cite{MR190017,MR299539,MR0476842}, Sch\"{u}tzenberger introduced an intriguing bijection on $ \call(P) $ called \emph{promotion}. Promotion has connections with various topics in algebraic combinatorics and representation theory, as seen in \cite{MR871081,MR3983097,MR2519848,MR2557880,MR2515772}. 

In \cite{DK20}, Defant and Kravitz extended the promotion map to an operator $ \partial:\Lambda(P)\to\Lambda(P) $, not necessarily invertible, that is defined on all of $ \Lambda(P) $. When the poset is a chain, extended promotion is dynamically equivalent to the bubble-sort map studied in \cite{Knuth1973}. Promotion can also be described in terms of Bender-Knuth involutions (first introduced by Haiman \cite{MR1158783} as well as Malvenuto and Reutenauer \cite{MR1297379}); in \cite{DK20}, Defant and Kravitz extended these Bender-Knuth involutions to arrive at an equivalent ``toggle'' definition of extended promotion. The following results of \cite{DK20} are crucial properties of extended promotion: \begin{enumerate}
	\item When restricted to $ \call(P) $, $ \partial $ agrees with Sch\"{u}tzenberger's promotion operator. 
	\item If $ L $ is a labeling of an $ n $-element poset $ P $, then $ \partial^{n-1}(L)\in\call(P) $. 
\end{enumerate} Thus, (extended) promotion\footnote{Henceforth, ``promotion'' always refers to $ \partial $ rather than its restriction to $ \call(P) $.} may be regarded as a sorting operator, where linear extensions are considered ``sorted.'' Property (2) shows that promotion sorts every labeling after at most $ n-1 $ applications.

We define the \emph{sorting time} of a labeling $ L $ to be the smallest $ k\in\N $ such that $ \partial^k(L)\in\call(P) $. Defant and Kravitz mainly studied \emph{tangled} labelings---those labelings with sorting time $ n-1 $. In particular, they enumerated these tangled labelings for a large class of posets called \emph{inflated rooted forests}. They also studied \emph{sortable} labelings---those labelings $ L $ such that $ \partial(L)\in\call(P) $---and enumerated these sortable labelings for arbitrary posets.

\subsection{Outline and Summary of Main Results}
In \Cref{Background}, we present the main definitions and background results needed for the rest of the paper. In \Cref{Promotion Fibers}, we study the cardinality of $ \partial\inverse(L) $ for an arbitrary labeling $ L $. Our \Cref{fiber} gives a formula for $ |\partial\inverse(L)| $ when $ P $ is a rooted tree poset. We also study the {degree of noninvertibility} of promotion and give a sharp lower bound for this (\Cref{noninvertibility bound}). \Cref{Quasi-Tangled} contains an explicit enumeration of the labelings with sorting time $ n-2 $ for a large class of posets called {inflated rooted trees with deflated leaves} (\Cref{quasi tangled enumeration}). A corollary of this result is that an $ n $-element poset with a unique minimal element with the property that the minimal element has exactly one parent has $ 2(n-1)!-(n-2)! $ quasi-tangled labelings. In \Cref{sorting time n-k-1}, we present an algorithmic approach to enumerating the labelings of a rooted tree poset with sorting time $ n-k-1 $ for fixed $ k\in\{1,\ldots,n-2\} $, and in \Cref{Tangled}, we make partial progress (\Cref{(n-1)!}) on the following conjecture:
\begin{conjecture}[\cite{DK20}, Conjecture 5.1]\label{conj: (n-1)!}
	If $ P $ is an $ n $-element poset, then $ P $ has at most $ (n-1)! $ tangled labelings. 
\end{conjecture} \noindent Finally, in \Cref{Open}, we present several open problems and further directions of inquiry.

\subsection{Extended Promotion}
Let $ P $ be an $ n $-element poset, and let $ L $ be a labeling of $ P $. For $ x\in P $ not maximal, the \emph{L-successor} of $ x $ is the element greater than $ x $ with minimal label. Now, let $ v_1=L\inverse(1) $. Let $ v_2 $ be the $ L $-successor of $ v_1 $; let $ v_3 $ the $ L $-successor of $ v_2 $, and so on until we get an element $ v_m $ that is maximal. The resulting chain $ v_1<_Pv_2<_P\cdots<_Pv_m $ is called the \emph{promotion chain} of $ L $. Now, define $ \partial(L) $ to be the labeling \[\partial(L)(x)=\begin{cases}
	L(x)-1 & \mathrm{if\ }x\not\in\{v_1,\ldots,v_m\};\\
	L(v_{i+1})-1 & \mathrm{if\ }x=v_i\ \mathrm{for\ }i\in\{1,\ldots,m-1\};\\
	n & \mathrm{if\ }x=v_m.
\end{cases}\] In other words, promotion may be thought of as decreasing each label by 1 (working modulo $ n $ so that $ 0=n $) and then cycling the promotion chain downwards one step. The following proposition captures a fundamental sorting property of promotion.

\begin{prop}[\cite{DK20}, Proposition 2.7]\label{sorting}
	If $ P $ is an $ n $-element poset, then $ \partial^{n-1}(\Lambda(P))=\call(P) $.
\end{prop}
\section{Preliminaries, Frozen Elements, and Some Special Classes of Posets}\label{Background}

\subsection{Preliminary Definitions}A \emph{lower order ideal} of a poset $ P $ is a subset $ Q\subset P $ such that for every $ x\in Q $ and $ y\in P $ with $ y<_Px $ we have $ y\in Q $. Similarly, an \emph{upper order ideal} of a poset $ P $ is a subset $ Q\subset P $ such that for all $ x\in Q $ and $ y\in P $ with $ y>_Px $ we have $ y\in Q $. It is often useful to note that $ Q $ is a lower order ideal of $ P $ if and only if $ P\setminus Q $ is an upper order ideal of $ P $. For $ x,y\in P $, we say that $ y $ \emph{covers} $ x $ and write $ x\lessdot y $ if $ x<_Py $ and $ \{z\in P\;|\;x<_Pz<_Py\}=\emptyset $. In this case, we say that $ y $ is a \emph{parent} of $ x $ and that $ x $ is a \emph{child} of $ y $. 

The \emph{Hasse diagram} of a poset $ P $ is a graphical illustration of its covering relations. Each element of $ P $ is represented by a vertex, and if $ x<_Py $, then the vertex corresponding to $ x $ is drawn below that corresponding to $ y $; there exists an edge between these vertices if and only if $ x\lessdot_P y $. We say a poset is \emph{connected} if its Hasse diagram is connected when regarded as a graph; the \emph{connected components} of $ P $ are the subposets induced by the connected components of the Hasse diagram of $ P $.

Suppose $ P $ is an $ n $-element poset, and let $ f:P\to\Z $ be an injective function. Then the \emph{standardization} of $ f $, denoted $ \st(f) $, is the labeling $ L:P\to[n] $ such that for all $ x,y\in P $, $ L(x)<L(y) $ if and only if $ f(x)<f(y) $. Note that this labeling is unique. Equivalently, if $ g:f(P)\to[n] $ is an order-preserving bijection, then $ \st(f)=g\circ f $.

\Cref{sorting} motivates the following definitions:
\begin{definition}
	Let $ L $ be a labeling of a poset $ P $. The \emph{sorting time} of $ P $ is the minimum number $ k\in\N $ such that $ \partial^k(L)\in\call(P) $.
\end{definition}
\begin{definition}
	For an $ n $-element poset $ P $, a labeling is called \emph{tangled} if it has sorting time $ n-1 $. A labeling is called \emph{quasi-tangled} if it has sorting time $ n-2 $.	We also say a labeling $ L $ is $ k $\emph{-promotion-sortable} (or just $ k $\emph{-sortable}) if $ \partial^k(L)\in\call(P) $. We call 1-sortable labelings \emph{sortable}. 
\end{definition}

We let the set of all tangled labelings of $ P $ be denoted by $ \T(P) $; we denote the set of all {sortable} labelings by $ \Sigma(P) $. If $ P $ is a poset, $ L $ a labeling of $ P $, and $ \gamma\in\N $, we also will frequently use the shorthand $ L_\gamma $ to denote $ \partial^\gamma(L) $. Note that $ L_0 =L $. 

\subsection{Frozen Elements and Some Useful Results About Promotion}
\begin{definition}
	Let $ L $ be a labeling of an $ n $-element poset $ P $. Define $ a $ to be the largest nonnegative integer less than or equal to $ n $ such that for each $ j\in\{n-a+1,\ldots,n\} $, the set $ \{x\in P\;|\;j\leq L(x)\leq n\} $ forms an upper order ideal in $ P $. An element $ x\in P $ is said to be \emph{frozen with respect to $ L $} if $ n-a+1\leq L(x)\leq n $. 
\end{definition} 
Note that $ L $ is a linear extension if and only if all elements of $ P $ are frozen. Also, we remark that it is possible for a labeling of an $ n $-element poset to have no frozen elements, namely, when $ L\inverse(n) $ is not maximal. 
\begin{example}
	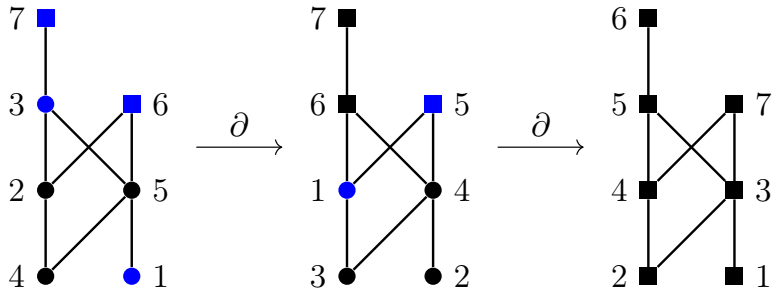
\begin{figure}[h!]
		\resizebox{0.65\textwidth}{!}{
			\begin{tikzpicture}[roundnode/.style={circle, inner sep=0pt, fill=black, minimum size=2mm}]
				
				\node[rectangle, inner sep=0pt, fill=blue, minimum size=2mm, fill=blue, label=left:{7}] (A7) at (0,3) {};
				\node[rectangle, inner sep=0pt, fill=blue, minimum size=2mm, label=right:{6}] (A6) at (1,2) {};
				\node[roundnode, label=right:{5}] (A5) at (1,1) {};
				\node[roundnode,fill=blue, label=right:{1}] (A3) at (1,0) {};
				\node[roundnode, fill=blue, label=left:{3}] (A4) at (0,2) {};
				\node[roundnode, label=left:{2}] (A2) at (0,1) {};
				\node[roundnode, label=left:{4}] (A1) at (0,0) {}; 
				
				\path (A1) edge [thick] (A2);
				\path (A2) edge [thick] (A4);
				\path (A4) edge [thick] (A7);
				\path (A3) edge [thick] (A5);
				\path (A5) edge [thick] (A6);
				\path (A4) edge [thick] (A5);
				\path (A6) edge [thick] (A2);
				\path (A1) edge [thick] (A5);
				
				\node[rectangle, inner sep=0pt, fill=black, minimum size=2mm, label=left:{7}] (B7) at (3.5,3) {};
				\node[rectangle, inner sep=0pt, fill=blue, minimum size=2mm, label=right:{5}] (B6) at (4.5,2) {};
				\node[roundnode, label=right:{4}] (B5) at (4.5,1) {};
				\node[roundnode, label=right:{2}] (B3) at (4.5,0) {};
				\node[rectangle, inner sep=0pt, fill=black, minimum size=2mm, label=left:{6}] (B4) at (3.5,2) {};
				\node[roundnode, fill=blue, label=left:{1}] (B2) at (3.5,1) {};
				\node[roundnode, label=left:{3}] (B1) at (3.5,0) {}; 
				
				\path (B1) edge [thick] (B2);
				\path (B2) edge [thick] (B4);
				\path (B4) edge [thick] (B7);
				\path (B3) edge [thick] (B5);
				\path (B5) edge [thick] (B6);
				\path (B4) edge [thick] (B5);
				\path (B6) edge [thick] (B2);
				\path (B1) edge [thick] (B5);
				
				\draw[->](1.75,1.5) -- (2.75,1.5) node[pos=0.5, above]{$ \partial $};
				
				\node[rectangle, inner sep=0pt, fill=black, minimum size=2mm, label=left:{6}] (C7) at (7,3) {};
				\node[rectangle, inner sep=0pt, fill=black, minimum size=2mm, label=right:{7}] (C6) at (8,2) {};
				\node[rectangle, inner sep=0pt, fill=black, minimum size=2mm, label=right:{3}] (C5) at (8,1) {};
				\node[rectangle, inner sep=0pt, fill=black, minimum size=2mm, label=right:{1}] (C3) at (8,0) {};
				\node[rectangle, inner sep=0pt, fill=black, minimum size=2mm, label=left:{5}] (C4) at (7,2) {};
				\node[rectangle, inner sep=0pt, fill=black, minimum size=2mm, label=left:{4}] (C2) at (7,1) {};
				\node[rectangle, inner sep=0pt, fill=black, minimum size=2mm, label=left:{2}] (C1) at (7,0) {}; 
				
				\path (C1) edge [thick] (C2);
				\path (C2) edge [thick] (C4);
				\path (C4) edge [thick] (C7);
				\path (C3) edge [thick] (C5);
				\path (C5) edge [thick] (C6);
				\path (C4) edge [thick] (C5);
				\path (C6) edge [thick] (C2);
				\path (C1) edge [thick] (C5);
				
				\draw[->](5.25,1.5) -- (6.25,1.5) node[pos=0.5, above]{$ \partial $};

		\end{tikzpicture}	}
		\caption{The above illustrates how extended promotion works. In the figure, the promotion chain is colored in blue, while the frozen elements are denoted using square-shaped nodes.}
	\end{figure}
\end{example}

Defant and Kravitz proved the following useful results about frozen elements.
\begin{lemma}[\cite{DK20}, Lemma 2.5]\label{frozen proper}
	Let $ L $ be a labeling of an $ n $-element poset $ P $. If $ F_0 $ denotes the set of frozen elements with respect to $ L $ and $ F_1 $ denotes the set of frozen elements with respect to $ L_1=\partial(L) $, then $ F_0 $ is properly contained in $ F_1 $ unless $ L\in\call(P) $.
\end{lemma}
\begin{lemma}[\cite{DK20}, Lemma 2.6]\label{frozen elements}
	With notation as above, for every $ 0\leq\gamma\leq n $, the elements $ L_\gamma\inverse(n-\gamma+1),\ldots,L_\gamma\inverse(n) $ are frozen with respect to $ L_\gamma $. 
\end{lemma}

Note that the $ \gamma=n-1 $ case of \Cref{frozen elements} immediately implies \Cref{sorting}.

\begin{lemma}\label{Swapnil lemma}
	Let $ P $ be an $ n $-element poset. For $ \gamma\in\N $ and any $ x\in\{2,\ldots,n\} $, we have that $ L_\gamma\inverse(x)\geq_PL_{\gamma+1}\inverse(x-1) $; equality holds if and only if $ L_\gamma\inverse(x) $ is not in the promotion chain of $ L_\gamma $. 
\end{lemma}
\begin{proof}
	Suppose first that $ L_\gamma\inverse(x) $ is not in the $ \gamma $th promotion chain. Then $ L_{\gamma+1}\inverse(x-1)=L_{\gamma}\inverse(x) $, and we are done. Otherwise, since $ x>1 $, there exists an element $ a<_PL_{\gamma}\inverse(x) $ such that $ a $ is in the promotion chain and $ L_{\gamma}\inverse(x) $ is the $ L_{\gamma} $-successor of $ a $. Hence, $ a=L_{\gamma+1}\inverse(x-1)<_PL_{\gamma}\inverse(x) $, as desired. 
\end{proof}

The following results will also be helpful later:
\begin{lemma}\label{inversion}
	Let $ x $ and $ y $ be two elements of $ P $ with $ y<_P x $. Fix some $ \gamma\in\{1,\ldots,n-2\} $, and suppose that $ L_\gamma(y)>L_\gamma(x) $. Then $ L_{\gamma}(y)=L_{\gamma-1}(y)-1 $. Moreover, we have that $ L_{\gamma-1}(y)>L_{\gamma-1}(x) $. 
\end{lemma}
\begin{proof}	
	Begin by letting $ a=L_\gamma(y) $ and $ b=L_{\gamma}(x) $, and note that $ a>b $. Observe that the first part of the lemma holds if and only if $ y $ is not in the $ (\gamma-1) $st promotion chain. Assume for a contradiction that $ y $ \emph{is} in the promotion chain of $ L_{\gamma-1} $. By \Cref{Swapnil lemma}, $ L_{\gamma-1}\inverse(b+1)\geq_P L_{\gamma}\inverse(b)=x>_Py $, so $ L_{\gamma-1}\inverse(a+1) $ cannot be the $ L_{\gamma-1} $-successor of $ y $, since $ a>b $. This is a contradiction. The second part of the lemma is simple: $ L_{\gamma-1}(x)\leq b+1<a+1=L_{\gamma-1}(y) $. 
\end{proof}

\begin{thm}[\cite{DK20}, Theorem 2.10]\label{k-untangled}
	For $ 0\leq k\leq n-2 $, an $ n $-element poset $ P $ has a labeling with sorting time $ n-k-1 $ if and only if it  has a lower order ideal of size $ k+2 $ that is not an antichain.
\end{thm}

\subsection{Rooted Tree and Forest Posets}
\begin{definition}
	A \emph{rooted forest poset} is a poset in which each element is covered by at most one other element. A \emph{rooted tree poset} is a connected rooted forest poset. Given a rooted tree poset, we say a subset $ Q $ of $ P $ is a subchain (respectively, subtree) if the poset induced by $ Q $ is a chain (respectively, tree). A \emph{rooted star poset} is a rooted tree poset where the root covers every leaf.
\end{definition} Note here that a rooted tree poset is a poset whose Hasse diagram is a rooted tree where the root is the unique maximal element. A rooted forest poset is a poset whose connected components are rooted tree posets. 

\section{Promotion Fibers}\label{Promotion Fibers}
Given a sorting procedure, it is a fundamental problem to enumerate the objects requiring only one iteration to become sorted. For promotion, Defant and Kravitz enumerated the sortable labelings of a poset in \cite{DK20}. It is then natural to turn our attention to the labelings that can be sorted in two applications of promotion. Note that this problem can be reduced to enumerating the preimages (under $ \partial $) of the sortable labelings.
\begin{definition}
	Let $ L $ be a labeling of a poset $ P $. We say an element $ x\in P $ is \emph{$ L $-golden} if for all $ y>_Px $, we have $ L(y)>L(x) $. A chain or tree is called $ L $-golden if all of its elements are $ L $-golden. If $ P $ is a rooted tree poset, an $ L $-golden subchain or subtree $ T $ is called \emph{maximal} if adding any other elements to $ T $ makes it no longer an $ L $-golden subchain or subtree.
\end{definition}
Given a rooted tree poset $ P $, its \emph{highest branch vertex} is the largest element $ x $ that covers more than one element. Let $ P $ be a rooted tree poset with root $ r $ and highest branch vertex $ b $. Note that the set $ \{x\;|\;b\leq_P x\leq_Pr\} $ forms a chain. Thus, given a labeling $ L $ of $ P $, there is a unique maximal $ L $-golden chain of elements greater than or equal to the highest branch vertex.
\begin{prop}[\cite{DK20}, Proposition 4.1]\label{L-golden}
	Let $ P $ be an $ n $-element poset. A labeling $ L $ of $ P $ is in the image of $ \partial $ if and only if $ L\inverse(n) $ is a maximal element of $ P $. If $ L\in\im(\partial) $, then $ |\partial\inverse(L)| $ is equal to the number of $ L $-golden chains of $ P $ containing $ L\inverse(n) $. 
\end{prop}
Let $ T $ be a rooted tree poset with root $ r $. Given $ x\in T $, let $ x=p_1\lessdot p_2\lessdot\cdots\lessdot p_{\omega(x)}=r $ be the path from $ x $ to the root. In our notation, $ \omega(x) $ gives the number of elements in this path.
\begin{thm}\label{fiber}
	Let $ P $ be a rooted tree poset with $ n $ elements, and fix a labeling $ L $ of $ P $. Let $ c_1<_P\cdots<_Pc_k $ be the maximal $ L $-golden chain of elements greater than or equal to the highest branch vertex. Define $ \mathcal{M} $ to be the set of all maximal $ L $-golden rooted subtrees of $ P $ whose roots are less than the maximal branch vertex. 
	If $ L $ is a labeling in $ \im(\partial) $ (i.e., if $ L\inverse(n) $ is maximal), then \[|\partial\inverse(L)|=2^{k-1}+\sum_{T\in\mathcal{M}}\sum_{x\in T}2^{\omega(x)+k-2},\] where $ \omega(x) $ is computed with respect to $ T $, not $ P $. 
\end{thm}
We start with a preparatory lemma before proving the theorem. 
\begin{lemma}\label{golden partition}
	With notation as in \Cref{fiber}, let $ M\in\partial\inverse(L) $, and let $ v_1<_P\cdots<_Pv_m $ be the promotion chain of $ M $. The promotion chain is an $ L $-golden chain containing $ L\inverse(n) $, and \[\{v_1,\ldots,v_m\}\subset\bigcup_{T\in\mathcal{M}}T\cup\{c_1,\ldots,c_k\}.\]
\end{lemma}
\begin{proof}
	Because $ P $ is a rooted tree poset and $ L\inverse(n) $ is maximal, the promotion chain of $ M $ is an $ L $-golden chain containing $ L\inverse(n) $. The lemma follows from the identity \[\bigcup_{T\in\mathcal{M}}T\cup\{c_1,\ldots,c_k\}=\bigcup_{C\in\mathcal{C}}C,\] where $ \mathcal{C} $ is the set of all $ L $-golden chains containing $ L\inverse(n) $. First, note that $ c_k=L\inverse(n) $ by assumption. Now, let $ G $ be some $ L $-golden chain containing $ L\inverse(n) $. Suppose that $ x_1<_P\cdots<_Px_i $ are the elements of $ G $ less than the highest branch vertex $ b $. These form an $ L $-golden subtree in $ \mathcal{M} $; therefore they are contained in some maximal $ L $-golden subtree of $ P $. The remaining elements of $ G $ form some subset of $ \{c_1,\ldots,c_k\} $. This shows the inclusion ``$ \supset $''; the reverse inclusion is immediate from the definitions. 
\end{proof}
\begin{proof}[Proof of \Cref{fiber}]
	We count the labelings $ M\in\partial\inverse(L) $. Again take $ \mathcal{C} $ to be the set of $ L $-golden chains containing $ L\inverse(n) $. By \Cref{L-golden}, we have $ |\partial\inverse(L)|=|\mathcal{C}|. $ Let $ \mathcal{G} $ be the set of $ L $-golden elements of $ P $. Note that every element in $ \mathcal{C} $ contains the root of $ P $. Let $ x\in\mathcal{G} $, and suppose $ x $ is the starting (lowest) element of $ C $. We count the $ C\in\mathcal{C} $ with lowest element $ x $. Every $ L $-golden element greater than $ x $ is possibly an element in $ C $; there are $ \omega(x)+k-1 $ such elements (take $ \omega(x)=0 $ when $ x\in\{c_1,\ldots,c_k\} $). Thus, there are $ 2^{\omega(x)+k-2} $ possible $ L $-golden chains starting with $ x $. Hence, \[|\mathcal{C}|=\sum_{x\in\mathcal{G}}2^{\omega(x)+k-2}.\]
	The theorem then follows from \Cref{golden partition} and from noting the following: when the promotion chain of $ M $ consists solely of elements greater than or equal to the highest branch vertex, choosing the promotion chain of $ M $ is just choosing some subset of $ \{c_1,\ldots,c_{k-1}\} $. 
\end{proof}
While the formula given in \Cref{fiber} is slightly more complicated than \[|\partial\inverse(L)|=\sum_{x\in\mathcal{G}}2^{\omega(x)+k-2},\] the more expanded form is preferable as it allows us to compute $ |\partial\inverse(L)| $ easily when $ P $ is a chain. Moreover, we also see immediately that \Cref{fiber} is a generalization of Corollary 2.4 of \cite{BCF}. \Cref{fiber} also lets us enumerate the 2-sortable labelings of rooted tree posets: 
\begin{corollary}
	Let $ P $ be an $ n $-element rooted tree poset. Let $ \Sigma(P)'=\Sigma(P)\cap\im(\partial) $ denote the set of sortable labelings of $ P $ in which $ L\inverse(n) $ is maximal. Then the number of 2-promotion-sortable labelings of $ P $ is \[\sum_{L\in\Sigma(P)'}|\partial\inverse(L)|=\sum_{L\in\Sigma(P)'}\left(2^{k-1}+\sum_{T\in\mathcal{M}}\sum_{x\in T}2^{\omega(x)+k-2}\right).\]
\end{corollary} 
\Cref{fiber} also yields a similar formula for the $ k $-sortable labelings for any $ k $ between 1 and $ n-1 $. However, such a formula would require being able to easily determine if a labeling is $ (k-1) $-sortable, which is not so easy for $ k>2 $, whereas determining if a labeling is sortable is simple. 

\Cref{fiber} also allows us to study the {degree of noninvertibility} of $ \partial $. In their 2020 paper \cite{DP20}, Defant and Propp defined the \emph{degree of noninvertibility} of a map $ f:X\to X $ to be \[\deg(f)=\frac{1}{|X|}\sum_{x\in X}|f\inverse(x)|^2.\] For any function $ f:X\to X $, we have that $ 1\leq\deg(f)\leq|X| $, and the lower bound is achieved when $ f $ is injective, whereas the upper bound is achieved when $ f $ is a constant map. More generally, the degree of a $ k $-to-1 map is $ k $. The following corollary can be proven na\"{i}vely---by simply computing that $ \partial $ is $ n $-to-1 when $ P $ is a rooted star---or by applying \Cref{fiber}: 
\begin{corollary}\label{star degree}
	If $ P $ is a rooted star poset with $ n $ elements, then $ \deg(\partial)=n $.
\end{corollary}
\begin{thm}\label{noninvertibility bound}
	Let $ P $ be a rooted tree poset with $ n $ elements. Then \[n\leq\deg(\partial),\] with equality if and only if $ P $ is a rooted star poset. 
\end{thm}
\begin{proof}
	Recall that a labeling $ L$ is in $\im(\partial) $ if and only if $ L\inverse(n) $ is maximal. Hence, $ |\im(\partial)|=(n-1)! $. The Cauchy-Schwarz inequality gives \[(n!)^2=\left(\sum_{L\in\im(\partial)}|\partial\inverse(L)|\right)^2\leq\sum_{L\in\im(\partial)}1^2\sum_{L\in\im(\partial)}|\partial\inverse(L)|^2=|\im(\partial)|n!\deg(\partial),\] and rearranging gives $ n\leq\deg(\partial) $, as desired. 
	
	For the second statement, recall that equality in the Cauchy-Schwarz inequality holds if and only if the two vectors are parallel. Hence, equality holds if and only if $ |\partial\inverse(L)| $ is the same for all $ L\in\im(\partial) $. This forces every $ L\in\im(\partial) $ to be a linear extension: Each labeling in $ \partial\inverse(L) $ corresponds to a $ L $-golden chain containing the root of $ P $. When $ L $ is linear, every chain containing the root is $ L $-golden. When $ L $ is not linear, there exists some element $ x $ that is not $ L $-golden. In particular, if $ r $ is the root of $ P$, $ x<_Pr $ is not an $ L $-golden chain, which reduces $ |\partial\inverse(L)| $, compared to the count for linear extensions. Thus, every labeling in $ \im(\partial) $ is linear, which forces $ P $ to be a rooted star poset.
\end{proof}

\section{Quasi-Tangled Labelings of Inflated Rooted Trees with Deflated Leaves}\label{Quasi-Tangled}
\subsection{Inflated Rooted Trees}
\begin{definition}
	Let $ Q $ be a finite poset. An \emph{inflation} of $ Q $ is a poset $ P $ along with a surjective map $ \varphi:P\to Q $ satisfying:\begin{enumerate}
		\item For all $ v\in Q $, the set $ \varphi\inverse(v) $ has a unique minimal element.
		\item If $ x,y\in P $ are such that $ \varphi(x)\neq\varphi(y) $, then $ x<_Py $ if and only if $ \varphi(x)<_Q\varphi(y) $.
	\end{enumerate}
	An \emph{inflated rooted tree poset} is an inflation of a rooted tree poset. An \emph{inflated rooted tree poset with deflated leaves} is an inflation of a rooted tree poset $ Q $ such that for all leaves $ \ell\in Q $ we have $ |\varphi\inverse(\ell)|=1 $. 
\end{definition}
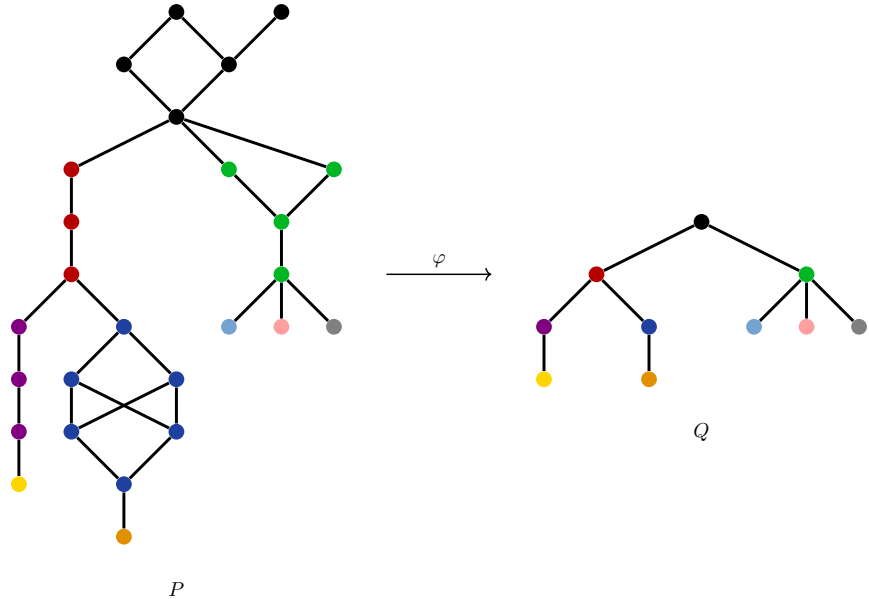
\begin{figure}[h!]
	\resizebox{0.7\textwidth}{!}{
		\begin{tikzpicture}[roundnode/.style={circle, inner sep=0pt, minimum size=3mm}]
			
			\node[roundnode, fill={rgb:yellow,5;red,1}] (I1) at (2,1) {};
			\node[roundnode, fill={rgb:red,3;blue,3}] (F1) at (2,4) {};
			\node[roundnode, fill={rgb:red,3;blue,3}] (F2) at (2,3) {};
			\node[roundnode, fill={rgb:red,3;blue,3}] (F3) at (2,2) {};
			\node[roundnode, fill={rgb:orange,5;black,1;yellow,2}] (J1) at (4,0) {};    
			\node[roundnode, fill={rgb:red,1;green,2;blue,5}] (G6) at (4,1) {};  
			\node[roundnode, fill={rgb:red,1;green,2;blue,5}] (G4) at (3,2) {};
			\node[roundnode, fill={rgb:red,1;green,2;blue,5}] (G5) at (5,2) {};
			\node[roundnode, fill={rgb:red,1;green,2;blue,5}] (G2) at (3,3) {};
			\node[roundnode, fill={rgb:red,1;green,2;blue,5}] (G3) at (5,3) {};
			\node[roundnode, fill={rgb:red,1;green,2;blue,5}] (G1) at (4,4) {};
			\node[roundnode, fill={rgb:red,5;black,2}] (B3) at (3,5) {};
			\node[roundnode, fill={rgb:red,5;black,2}] (B2) at (3,6) {};
			\node[roundnode, fill={rgb:red,5;black,2}] (B1) at (3,7) {};
			\node[roundnode, fill={rgb:red,3;white,5}] (D1) at (7,4) {};
			\node[roundnode, fill={rgb:blue,4;white,5;green,2}] (D2) at (6,4) {};
			\node[roundnode, fill={rgb:black,5;white,5}] (D3) at (8,4) {};
			\node[roundnode, fill={rgb:green,5;blue,1;black,1}] (C4) at (7,5) {};
			\node[roundnode, fill={rgb:green,5;blue,1;black,1}] (C3) at (7,6) {}; 
			\node[roundnode, fill={rgb:green,5;blue,1;black,1}] (C1) at (6,7) {};
			\node[roundnode, fill={rgb:green,5;blue,1;black,1}] (C2) at (8,7) {};
			\node[roundnode, fill=black] (A5) at (5,8) {};
			\node[roundnode, fill=black] (A3) at (4,9) {};
			\node[roundnode, fill=black] (A4) at (6,9) {};
			\node[roundnode, fill=black] (A2) at (5,10) {};
			\node[roundnode, fill=black] (A1) at (7,10) {}; 
			\node (P) at (5,-1) {$ P $};
			
			\path (A1) edge [ultra thick] (A4);
			\path (A2) edge [ultra thick] (A3);
			\path (A2) edge [ultra thick] (A4);
			\path (A3) edge [ultra thick] (A5);
			\path (A4) edge [ultra thick] (A5);
			\path (A5) edge [ultra thick] (B1);
			\path (A5) edge [ultra thick] (C1);
			\path (A5) edge [ultra thick] (C2);
			\path (C1) edge [ultra thick] (C3);
			\path (C2) edge [ultra thick] (C3);
			\path (C3) edge [ultra thick] (C4);
			\path (C4) edge [ultra thick] (D1);
			\path (C4) edge [ultra thick] (D2);
			\path (C4) edge [ultra thick] (D3);
			\path (B1) edge [ultra thick] (B2);
			\path (B2) edge [ultra thick] (B3);
			\path (B3) edge [ultra thick] (F1);
			\path (B3) edge [ultra thick] (G1);
			\path (F1) edge [ultra thick] (F2);
			\path (F2) edge [ultra thick] (F3);
			\path (F3) edge [ultra thick] (I1);
			\path (G1) edge [ultra thick] (G2);
			\path (G1) edge [ultra thick] (G3);
			\path (G2) edge [ultra thick] (G4);
			\path (G2) edge [ultra thick] (G5);
			\path (G3) edge [ultra thick] (G4);
			\path (G3) edge [ultra thick] (G5);
			\path (G5) edge [ultra thick] (G6);
			\path (G4) edge [ultra thick] (G6);
			\path (G6) edge [ultra thick] (J1);
			
			\node[roundnode, fill=black] (R1) at (15,6) {};
			\node[roundnode, fill={rgb:red,5;black,2}] (R2) at (13,5) {};
			\node[roundnode, fill={rgb:green,5;blue,1;black,1}] (R3) at (17,5) {};
			\node[roundnode, fill={rgb:red,3;blue,3}] (R4) at (12,4) {};
			\node[roundnode, fill={rgb:red,1;green,2;blue,5}] (R5) at (14,4) {};    
			\node[roundnode, fill={rgb:blue,4;white,5;green,2}] (R6) at (16,4) {};  
			\node[roundnode, fill={rgb:red,3;white,5}] (R7) at (17,4) {};
			\node[roundnode, fill={rgb:black,5;white,5}] (R8) at (18,4) {};
			\node[roundnode, fill={rgb:yellow,5;red,1}] (R9) at (12,3) {};
			\node[roundnode, fill={rgb:orange,5;black,1;yellow,2}] (R10) at (14,3) {};
			\node (Q) at (15,2) {$ Q $};
			
			\path (R1) edge [ultra thick] (R2);
			\path (R1) edge [ultra thick] (R3);
			\path (R2) edge [ultra thick] (R4);
			\path (R2) edge [ultra thick] (R5);
			\path (R3) edge [ultra thick] (R6);
			\path (R3) edge [ultra thick] (R7);
			\path (R3) edge [ultra thick] (R8);
			\path (R4) edge [ultra thick] (R9);
			\path (R5) edge [ultra thick] (R10);
			
			\draw[->, thick](9,5) -- (11,5) node[pos=0.5, above]{$ \vphi $};

	\end{tikzpicture}		}
	\caption{\label{fig:inflated rooted tree}$ (P,\vphi) $ is an inflated rooted tree with deflated leaves, where $ P $ inflates $ Q $. The colors illustrate the preimages of the elements of $ Q $. }
\end{figure}

\begin{remark}\label{reduced}
	We say that a rooted tree poset is \emph{reduced} if no vertex has exactly one child. Since the composition of inflations is also an inflation, we see that every inflated rooted tree poset is the inflation of a reduced rooted tree poset. When we refer to an inflation $ \vphi:P\to Q $ of a rooted tree poset with deflated leaves we will assume that the subposet of $ Q $ obtained by removing its leaves is reduced.
\end{remark}

\subsection{Setup and Statement of the Main Theorem}
In the following, we give a formula enumerating the quasi-tangled labelings of inflated rooted tree posets with deflated leaves. The following is shown in the proof of \Cref{k-untangled}, but we state it as its own lemma here:
\begin{lemma}[\cite{DK20}, Proof of Theorem 2.10]\label{lower order ideal of size 3}
	Let $ P $ be an $ n $-element poset and $ L $ a labeling of $ P $ such that $ L_{n-k}\not\in\call(P) $. Then $ \{L_{n-k}\inverse(1),\ldots,L_{n-k}\inverse(k)\} $ forms a lower order ideal of size $ k $ that is not an antichain, and the restriction of $ L_{n-k} $ to this set is not a linear extension. 
\end{lemma}

\begin{lemma}\label{L^-1(n) position}
	If $ L $ is a quasi-tangled labeling of an $ n $-element inflated rooted tree poset $ P $ with deflated leaves, then one of the following holds:
	\begin{enumerate}
		\item $ L\inverse(n-1) $ is minimal;
		\item $ L\inverse(n) $ is minimal;
		\item $ L\inverse(n) $ covers a minimal element. 
	\end{enumerate} In each case, the element in question is involved in an inversion after $ n-3 $ promotions.
\end{lemma}
\begin{proof}
	By \Cref{lower order ideal of size 3}, the set $ Y=\{L_{n-3}\inverse(1),L_{n-3}\inverse(2),L_{n-3}\inverse(3)\} $ forms a lower order ideal that is not an antichain. Moreover, $ L_{n-3} $ restricted to this set is not a linear extension.
	
	Now, note that every non-antichain lower order ideal $ I $ with three elements occurring in an inflated rooted tree poset with deflated leaves is one of the following: \begin{enumerate}
		\item $ I $ is a chain with 3 elements;
		\item $ I $ is the rooted tree poset on 3 elements with 2 leaves;
		\item $ I $ is the disjoint union of a singleton and a chain of size 2.
	\end{enumerate}
	
	Suppose $ Y $ is of type (1). If $ L_{n-3}\inverse(3) $ is maximal in $ I $, then $ L_{n-3}\inverse(2)\lessdot L_{n-3}\inverse(1)\lessdot L_{n-3}\inverse(3) $, and repeatedly applying \Cref{inversion} tells us that $ L\inverse(n-1) $ is minimal. If $ L_{n-3}\inverse(3) $ is not maximal, then it covers a minimal element or is minimal itself. We may again repeatedly apply \Cref{inversion} to see that $ L\inverse(n) $ either covers a minimal element or is minimal. 
	
	Suppose $ Y $ is of type (2). It follows that $ L_{n-3}\inverse(3) $ cannot be maximal, because then $ L_{n-3} $ would be a linear extension. So, $ L_{n-3}\inverse(3) $ is minimal. The element covering $ L_{n-3}\inverse(3) $ has a smaller label in $ L_{n-3} $, so $ L\inverse(n) $ is minimal by \Cref{inversion}.
	
	Lastly, suppose $ Y $ is of type (3). If $ L_{n-3}\inverse(3) $ occupies its own component, then $ L_{n-3}\inverse(2)\lessdot L_{n-3}\inverse(1) $ and $ L\inverse(n-1) $ is minimal. Otherwise, $ L_{n-3}\inverse(3) $ is covered by either $ L_{n-3}\inverse(1) $ or $ L_{n-3}\inverse(2) $, so $ L\inverse(n) $ must be minimal. 
\end{proof}

\Cref{quasi tangled enumeration} enumerates the quasi-tangled labelings of inflated rooted trees with deflated leaves. In particular, note that this relatively large class of posets includes rooted tree posets. While this particular class of posets seems artificial, it will be important that the posets we are working with have the property that if one removes a minimal element or an element covering a minimal element from the poset, then the resulting poset remains an inflated rooted tree. 

We remark here that the quasi-tangled labelings of a poset are those labelings $ L $ such that $ L_{n-3} $ is not a linear extension but $ L_{n-2} $ is a linear extension, i.e., $ L $ is not tangled. To count these labelings, we condition on the positions of $ L\inverse(n-1) $ and $ L\inverse(n) $ and keep track of where $ L_{\gamma}\inverse(n-1-\gamma) $ and $ L_\gamma\inverse(n-2-\gamma) $ end up (after $ n-3 $ promotions these are the elements that are labeled 2 and 1, respectively). In particular, we consider a uniformly random labeling $ L $ of $ P $ with $ L\inverse(n) $ or $ L\inverse(n-1) $ fixed and calculate the probability that, at each ``branch vertex,'' the elements in question ``get pulled down'' in the desired direction, i.e., in the direction such that $ L_{n-3} $ is not a linear extension. The proof of the theorem involves a lot of casework, which we prove beforehand in several technical lemmas.

Let $ (P,\varphi) $ be an inflation of a rooted tree poset $ Q $. For each leaf $ \ell $ of $ Q $, there exists a unique path in the Hasse diagram of $ Q $ from $ \ell $ to the root. Let the elements of this path be called $ u_{\ell,0},u_{\ell,1},\ldots,u_{\ell,\omega(\ell)} $, where $ u_{\ell,0}=\ell $ and for all $ 1\leq j\leq\omega(\ell) $, $ u_{\ell,j} $ covers $ u_{\ell,j-1} $. Then define \[b_{\ell,j}=\sum_{v\leq_Q u_{\ell,j-1}}|\varphi\inverse(v)|\qquad\mathrm{and}\qquad c_{\ell,j}=\sum_{v<_Q u_{\ell,j}}|\varphi\inverse(v)|.\] Note that $ \frac{b_{\ell,j}}{c_{\ell,j}} $ is the fraction of elements below the minimal element of $ \varphi\inverse(u_{\ell,j}) $ that ``lie in the direction'' of $ \varphi\inverse(\ell) $. The following enumerates the number of tangled labelings of $ P $:
\begin{thm}[\cite{DK20}, Theorem 3.5]\label{tangled enumeration}
	If $ P $ is an inflation of a rooted tree poset $ Q $, then the number of tangled labelings of $ P $ is \[(n-1)!\sum_{i=1}^{s}\prod_{j=1}^{\omega(i)}\frac{b_{i,j}-1}{c_{i,j}-1},\] where $ s $ is the number of leaves of $ Q $. 
\end{thm}

Now, define $ M $ to be the set of minimal elements of $ P $. We define subsets $ R,S,T $ of $ M $: Let $ \ell\in M $, and let $ x $ cover $ \ell $.  \begin{enumerate}
	\item Put $ \ell$ in $ R $ if $ \ell $ is the only child of $ x $ and the parent of $ x $ exists and has multiple children.
	\item Put $ \ell$ in $ S $ if $ x $ has precisely two children;
	\item Put $ \ell$ in $ T $ if $ \ell $ is the only child of $ x $, the parent of $ x $ exists, and $ x $ is the only child of its parent.
\end{enumerate}
Note here that $ R $, $ S $, and $ T $ are disjoint but do not partition $ P $. However, it is not difficult to see that $ R $, $ S $, and $ T $ partition the set of minimal elements of $ P $ that, along with their parents, lie in non-antichain lower order ideals of size 3. See \Cref{fig: min elts}.
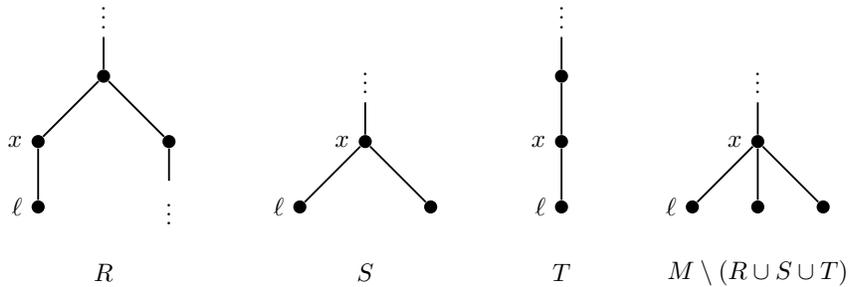
\begin{figure}[h!]
	\resizebox{0.69\textwidth}{!}{\begin{tikzpicture}[roundnode/.style={circle, inner sep=0pt, fill=black, minimum size=2mm}]
			
			\node[roundnode, label=left:{$ \ell $}] (R1) at (0,0) {};
			\node[roundnode, label=left:{$ x $}] (R2) at (0,1) {};
			\node[roundnode, label=left:{}] (R3) at (1,2) {};
			\node[roundnode] (R4) at (2,1) {};
			\node (R5) at (2,0) {\vdots};
			\node (R) at (1,3) {\vdots};
			\node (R6) at (1,-1) {$ R $};
			
			\path (R1) edge [thick] (R2);
			\path (R2) edge [thick] (R3);
			\path (R3) edge [thick] (R4);
			\path (R4) edge [thick] (R5);
			\path (R3) edge [thick] (R);
			
			\node[roundnode, label=left:{$ \ell $}] (S1) at (4,0) {};
			\node[roundnode, label=left:{$ x $}] (S2) at (5,1) {};
			\node[roundnode] (S3) at (6,0) {};
			\node (S) at (5,2) {\vdots};
			\node (S4) at (5,-1) {$ S $};
			
			\path (S1) edge [thick] (S2);
			\path (S2) edge [thick] (S3);
			\path (S2) edge [thick] (S);
			
			\node[roundnode, label=left:{$ \ell $}] (T1) at (8,0) {};
			\node[roundnode, label=left:{$ x $}] (T2) at (8,1) {};
			\node[roundnode, label=left:{}] (T3) at (8,2) {};
			\node (T4) at (8,-1) {$ T $};
			\node (T) at (8,3) {\vdots};
			
			\path (T1) edge [thick] (T2);
			\path (T2) edge [thick] (T3);
			\path (T3)edge [thick] (T);
			
			\node[roundnode, label=left:{$ \ell $}] (U1) at (10,0) {};
			\node[roundnode, label=left:{$ x $}] (U2) at (11,1) {};
			\node[roundnode] (U3) at (12,0) {};
			\node[roundnode] (U4) at (11,0) {};
			\node (U5) at (11,-1) {$ M\setminus(R\cup S\cup T) $};
			\node (U) at (11,2) {\vdots};
			
			\path (U1) edge [thick] (U2);
			\path (U2) edge [thick] (U3);
			\path (U2) edge [thick] (U4);
			\path (U2) edge [thick] (U);

	\end{tikzpicture}}
	\caption{\label{fig: min elts}An illustration of what the elements of $ R, $ $ S, $ and $ T $ ``look like.''}
\end{figure}
\begin{thm}\label{quasi tangled enumeration}
	Let $ (P,\vphi) $ be an inflation of the rooted tree poset $ Q $ with deflated leaves. Let $ M $, $ R $, $ S $, $ T $, and the $ u_{\ell,j} $'s, $ b_{\ell,j} $'s, and $ c_{\ell,j} $'s be as defined above. Then the number of quasi-tangled labelings of $ P $ is \begin{equation*}\resizebox{\textwidth}{!} 
		{
			$(n-1)!\left(2\displaystyle\sum_{\ell\in T}\prod_{j=2}^{\omega(\ell)}\frac{b_{\ell,j}-1}{c_{\ell,j}-1}-\frac{1}{n-1}\sum_{\ell\in T}\prod_{j=2}^{\omega(\ell)}\frac{b_{\ell,j}-2}{c_{\ell,j}-2}+2\sum_{\ell\in R}\prod_{j=2}^{\omega(\ell)}\frac{b_{\ell,j}-1}{c_{\ell,j}-1}+\sum_{\ell\in S}\prod_{j=2}^{\omega(\ell)}\frac{(b_{\ell,j}-1)(b_{\ell,j}-2)}{(c_{\ell,j}-1)(c_{\ell,j}-2)}\right).$
		}
	\end{equation*}
\end{thm}
\begin{corollary}
	Let $ P $ be a poset with $ n $ elements. Suppose that $ P $ has a unique minimal element and that this minimal element has exactly one parent. Then $ P $ has $ 2(n-1)!-(n-2)! $ quasi-tangled labelings. 
\end{corollary}
\subsection{Computing Probabilities}

In this section, we compute several probabilities that will help us in the proof of \Cref{quasi tangled enumeration}. Importantly, in \Cref{probability lemma}, we generalize Lemma 3.11 of \cite{DK20}, which tells us that the probability of a certain label ending up in some subtree of our inflated rooted tree poset after a certain number of promotions is proportional to the size of the subtree.
\begin{lemma}[\cite{DK20}, Lemma 3.9]\label{promotion dependence}
	Let $ P $ be an $ N $-element poset, and let $ X=\{y\in P\;|\;y<_Px\} $ for some $ x\in P $. Suppose every element of $ P $ that is comparable with some element of $ X $ is also comparable with $ x $. If $ L $ and $ \tilde{L} $ are labelings of $ P $ that agree on $ P\setminus X $, then for every $ \gamma\geq1 $, the labelings $ L_\gamma $ and $ \tilde{L}_\gamma $ also agree on $ P\setminus X $.
\end{lemma}
\begin{lemma}[\cite{DK20}, Lemma 3.10]\label{label distribution}
	Let $ P $, $ x $, and $ X $ be defined as in \Cref{promotion dependence}. If $ L $ is a labeling of $ P $ and $ \gamma\geq0 $, then the set $ L_\gamma(X) $ depends only on the set $ L(X) $ and the restriction $ L_{P\setminus X} $; it does not depend on the way in which the labels in $ L(X) $ are distributed among the elements of $ X $. 
\end{lemma}
For the rest of this subsection, let $ P $, $ x $, and $ X $ be defined as in \Cref{promotion dependence}.  Suppose there is a partition of $ X $ into disjoint subsets $ A $ and $ B $ such that no element of $ A $ is comparable to an element of $ B $. Note that both $ A $ and $ B $ are lower order ideals of $ P $. 

\begin{definition}\label{pulls down}
	Let $ L $ be a labeling of $ P $. Suppose that $ k\in[N-1] $ and $ m\in[N] $ are such that $ m+k\leq N $, $ L_k\inverse(m)\in X $, and $ L\inverse(m+k)\not\in X $. We say that $ \gamma $ \emph{pulls down} $ m+k $ if $ \gamma<k $ is the largest index such that $ L_{\gamma}\inverse(m+k-\gamma)\not\in X $. We see immediately from the definition that $ \gamma $ is the unique such index pulling down $ m+k $ and that $ \gamma $ pulls down exactly one label.
\end{definition}

With notation as above, note that in order to determine whether $ L_k\inverse(m)\in A $, it suffices to determine whether $ L_{\gamma}\inverse(1)\in A $: because $ L_{\gamma}\inverse(m+k-\gamma)\not\in X $ and $ L_{\gamma+1}\inverse(m+k-\gamma-1)\in X $, we have that $ L_{\gamma}\inverse(m+k-\gamma) $ is in the $ \gamma $th promotion chain. Hence, $ L_{\gamma+1}\inverse(m+k-\gamma-1)\in A $ if and only if $ L_{\gamma}\inverse(1)\in A $ (recall $ A $ and $ B $ are disjoint lower order ideals). 

To each such $ m+k $, one can associate a decreasing sequence of indices $ \gamma_0,\gamma_1,\ldots,\gamma_r $ whose values depend only on $ L|_{P\setminus X} $ as follows. Let $ \gamma_0 $ pull down $ m+k $. By \Cref{label distribution}, the value $ \gamma_0 $ depends only on $ L|_{P\setminus X} $. If $ L\inverse(\gamma_0+1)\in X $, we are done. Otherwise, let $ \gamma_1 $ pull down $ \gamma_0+1 $, where we let $ \gamma_0 $ take the role of $ k $ and $ 1 $ the role of $ m $ (note that $ \gamma_1<\gamma_0 $). If $ L\inverse(\gamma_1+1)\in X $, we are done; otherwise, let $ \gamma_2 $ pull down $ \gamma_1+1 $. This process can be continued, where $ \gamma_{i+1} $ pulls down $ \gamma_i+1 $. Since $ 0\leq\gamma_{i+1}<\gamma_i $, this process eventually terminates, yielding a decreasing sequence $ \gamma_0,\ldots,\gamma_r $. By \Cref{label distribution}, we see that the values $ \gamma_0,\ldots,\gamma_r $ depend only on the set $ L(X) $ and $ L|_{P\setminus X} $, not on the way in which the labels in $ L(X) $ are distributed. 

Let $ a,b\in[N] $ and $ k\in[N-1] $ be such that $ a+k,b+k\leq N $, $ L_k\inverse(a),L_k\inverse(b)\in X $, and $ L\inverse(a+k),L\inverse(b+k)\not\in X $. Suppose $ a\neq b $. Let $ \alpha_0,\ldots,\alpha_r $ and $ \beta_0,\ldots,\beta_s $ be the sequences associated to $ a+k $ and $ b+k $, respectively. We claim that $ \{\alpha_0,\ldots,\alpha_r\}\cap\{\beta_0,\ldots,\beta_s\}=\emptyset $. Assume the contrary. If $ \alpha_i=\beta_j $ for some $ 0\leq i\leq r $ and $ 0\leq j\leq s $, then it follows from \Cref{pulls down} that $ \alpha_r=\beta_s $. 

Without loss of generality, suppose $ r\leq s $. If $ r<s $, then by definition, $ \alpha_r=\beta_s $ pulls down $ \alpha_{r-1}+1=\beta_{s-1}+1 $, $ \alpha_{r-1}=\beta_{s-1} $ pulls down $ \alpha_{r-2}+1=\beta_{s-2}+1 $, etc., until $ \alpha_0=\beta_{s-r} $ pulls down $ a+k=\beta_{s-r-1}+1 $. It follows that $ a+k-1=\beta_{s-r-1} $. Note that $ a+k-1\geq k $, but $ \beta_{s-r-1}\leq\beta_0<k $. This is a contradiction. The case where $ s>r $ is identical. If $ r=s $, then $ \alpha_0=\beta_0 $, and it follows that $ a+k=b+k $, contradicting our assumption that $ a\neq b $. 

The following is a generalization of Lemma 3.11 in \cite{DK20}. For $ d=1 $, the lemmas are exactly the same. This lemma will be applied repeatedly in the proof of \Cref{quasi tangled enumeration}. Informally, it states that given a list of $ d $ labels whose corresponding elements are in $ X $ after $ k $ promotions, the probability that all of these labels are in $ A $ is proportional to $ (|A|)!/(|A|-d-1)! $. 

\begin{lemma}[Probability Lemma]\label{probability lemma}
	Let $ P $, $ x $, and $ X $ be defined in \Cref{promotion dependence}, and let $ A $ and $ B $ be defined as above. Let $ k\in[N-1] $, and let $ n_1,\ldots,n_d\in[N] $ be such that $ n_i+k\leq N $ for all $ 1\leq i\leq d $. Fix an injective map $ M:P\setminus X\to[N] $ such that every labeling $ L $ extending $ M $ has the property that $ L_k\inverse(n_1),\ldots,L_k\inverse(n_d)\in X $. If such an $ L $ is chosen uniformly at random among all such extensions of $ M $, then the probability that $ L_k\inverse(n_1),\ldots,L_k\inverse(n_d)\in A $ is \[\frac{|A|(|A|-1)\cdots(|A|-d)}{|X|(|X|-1)\cdots(|X|-d)}.\] 
\end{lemma}
\begin{proof}
	Suppose that $ n_1<n_2<\cdots<n_d $, and let $ n_{i_1}<\cdots<n_{i_t} $ be the subset of labels such that $ L\inverse(n_{i}+k)\not\in X $. For all $ s\not\in\{i_1,\ldots,i_t\} $, because $ L\inverse(n_s+k)\in X $, \Cref{Swapnil lemma} gives that $ L_k\inverse(n_s)\in A $ if and only if $ L\inverse(n_s+k)\in A $.
	
	Now, by our discussion above, to each $ n_{i_j} $ we may associate a decreasing sequence of $ \gamma^j $'s given by $ \gamma^j_0>\cdots>\gamma^j_{r(j)} $. Note that by \Cref{label distribution}, the set of $ \gamma^j $'s depends only on $ M $, not on how the labels in $ L(X) $ are distributed. Recall that the sets $ \{\gamma^j_0,\ldots,\gamma^j_{r(j)}\} $ are pairwise disjoint. Importantly, we have that the $ \gamma^j_{r(j)} $'s are all distinct. 
	
	Fix some $ n_{j} $ for $ j\in\{i_1,\ldots,i_t\} $. In this and the next paragraph, denote the associated sequence by $ \gamma_0,\ldots,\gamma_r $. We claim that the probability that $ L_k\inverse(n_{j})\in A $ is equal to the probability that $ L_{\gamma_{r}}\inverse(1)\in A $. To see why this is true, recall that $ \gamma_0 $ pulls down $ n_{j}+k $. In the discussion above, we showed that $ L_k\inverse(n_{j})\in A $ if and only if $ L_{\gamma_{0}}\inverse(1)\in A $. Now, $ \gamma_1 $ pulls down $ \gamma_0+1 $, so $ L_{\gamma_0}\inverse(1)\in A $ if and only if $ L_{\gamma_1}\inverse(1)\in A $. Clearly, we may continue in this manner until we see that $ L_{k}\inverse(n_{j})\in A $ if and only if $ L_{\gamma_{r}}\inverse(1)\in A $. 
	
	By assumption, $ L\inverse(\gamma_{r}+1)\in X $. Because $ A $ and $ B $ are disjoint lower order ideals, \Cref{Swapnil lemma} implies that $ L_{\gamma_{r}}\inverse(1)\in A $ if and only if $ L\inverse(\gamma_{r}+1)\in A $. Hence, the probability that $ L_k\inverse(n_{j})\in A $ is equal to the probability that $ L\inverse(\gamma_{r}+1)\in A $.
	
	Thus, for all $ i $, we have reduced calculating the probability that $ L_k\inverse(n_i)\in A $ to calculating the probability that $ L\inverse(a_i)\in A $ for some particular label $ a_i\in[N] $. Note that our assumptions on $ M $ give $ L\inverse(a_i)\in X $ for all $ i $. For each $ n_{i_j} $, we have that $ a_{i_j}=\gamma^j_{r(j)}+1 $. Recall that the $ \gamma_{r(j)}^j $'s are all distinct; moreover note that for all $ s\not\in\{i_1,\ldots,i_t\} $, we have that $ \gamma^j_{r(j)}+1\neq n_s+k $, since $ \gamma^j_{r(j)}+1\leq\gamma^j_0+1\leq k<n_s+k $. For each $ s\not\in\{i_1,\ldots,i_t\} $, $ a_s=n_s+k $. Thus, the $ a_i $'s are distinct. Since $ L $ is chosen uniformly at random from the labelings extending $ M $, it follows that the probability $ L_{k}\inverse(n_i)\in A $ for all $ i $ is \[\frac{|A|(|A|-1)\cdots(|A|-d)}{|X|(|X|-1)\cdots(|X|-d)},\] as desired.	
\end{proof}

\subsection{Proof of the Main Theorem}
\begin{lemma}\label{standardization}
	Let $ P $ be an $ n $-element poset, and let $ L $ be a labeling of $ P $. Let $ x_0\in P $. Define $ \tilde{P}=P\setminus\{x_0\} $ and $ \tilde{L}=\st(L|_{\tilde{P}}) $. Suppose that $ x_0 $ is not part of the promotion chain for any of the first $ \gamma $ promotions. Then $ \st(L_\gamma|_{\tilde{P}})=\tilde{L}_{\gamma} $.
\end{lemma}
\begin{proof}
	Recall that promotion depends only on the promotion chain, which in turn depends only on the relative order of the labels. Since $ x_0 $ is never in the promotion chain for the first $ \gamma $ promotions, the promotion chains of $ L_0,\ldots,L_{\gamma} $ and $ \tilde{L}_0,\ldots,\tilde{L}_{\gamma} $ are the same, as desired.
\end{proof}
Before proving \Cref{quasi tangled enumeration}, we define some notation. Let $ P $ be as in \Cref{quasi tangled enumeration}, and let $ x_0\in P $ either cover a unique minimal element or be minimal itself. If $ x_0 $ is minimal, let $ \ell=x_0 $; if it covers a minimal element, denote this minimal element by $ \ell $. Define $ \tilde{P} $ and $ \tilde{L} $ as in \Cref{standardization}, and let $ \tilde{\vphi}=\vphi|_{\tilde{P}} $. Note that $ \omega(\ell)\geq1 $ unless $ Q $ is a one-element poset; if $ Q $ has only one element, then so does $ P $. A one-element poset has no quasi-tangled labelings, so henceforth we assume $ Q $ has more than one element. 

For $ j\in\{2,\ldots,\omega(\ell)\} $, let $ x_j $ be the minimal element of $ \tilde{\vphi}\inverse(u_{\ell,j}) $. Also define \[X_j=\{y\in\tilde{P}\;|\;y<_{\tilde{P}}x_j\}\quad\mathrm{and}\quad A_j=\bigcup_{v\leq_Qu_{\ell,j-1}}\tilde{\varphi}\inverse(v).\] Let $ A_j' $ be defined analogously but with $ \vphi $ instead of $ \tilde{\vphi} $. Recall that for $ j\in\{2,\ldots,\omega(\ell)\} $, \[b_{\ell,j}=\sum_{v\leq_Q u_{\ell,j-1}}|\varphi\inverse(v)|\qquad\mathrm{and}\qquad c_{\ell,j}=\sum_{v<_Q u_{\ell,j}}|\varphi\inverse(v)|,\] where $ u_{\ell,0},u_{\ell,1},\ldots,u_{\ell,\omega(\ell)} $ is the unique path in $ Q $ from $ \vphi(\ell) $ to the root. 

In order to count the quasi-tangled labelings of $ P $, we condition on the label of $ x_0 $ and count the labelings $ L $ such that $ L_{n-3}\not\in\call(P) $ and there exists $ y\in P $ such that $ y>_P x_0 $ and $ L_{n-3}(y)<L_{n-3}(x_0) $. For example, when $ x_0 $ is minimal, we count the labelings $ L $ such that $ L(x_0)=n-1 $ and $ L_{n-3}\inverse(1)>_Px_0=L_{n-3}\inverse(2) $. Note here that $ L_{n-3}\inverse(1)>_PL_{n-3}\inverse(2) $ if and only if $ L_{n-3}\inverse(1)\in A_2' $. Our strategy is to fix the label of $ x_0 $ and choose a labeling uniformly at random among the $ (n-1)! $ such labelings of $ P $; observe that this induces the uniform distribution on the labelings of $ \tilde{P} $. Given such a random labeling, we want to calculate the probability that certain labels end up in $ A_2' $.

We will show later that, in each case, calculating this probability can be reduced to calculating the probability that the labels in question end up in $ A_2 $. Thus, we make the following definitions: Let $ K $ be some nonempty subset of $ \{1,2\} $. For $ x_0 $ and $ \ell $ as defined above and $ j\in\{2,\ldots,\omega(\ell)\} $, let $ E_{\ell,j} $ be the event that $ K\subset \tilde{L}_{n-3}(A_j) $. In other words, $ E_{\ell,j} $ is the event that every label in $ K $ ends up on the ``correct side'' of $ x_j $ after $ n-3 $ promotions. We would like to compute $ \Prob(E_{\ell,2}) $ for $ \tilde{L} $. To do so, we note that \begin{equation}\label{conditional prob}
	\Prob(E_{\ell,2})=\Prob(E_{\ell,\omega(\ell)})\Prob(E_{\ell,\omega(\ell)-1}\,|\,E_{\ell,\omega(\ell)})\cdots\Prob(E_{\ell,2}\,|\,E_{\ell,3})
\end{equation} and compute the multiplicands on the right-hand side of the equation above. 

\begin{figure}[h!]
	\resizebox{0.75\textwidth}{!}{
		\begin{tikzpicture}[roundnode/.style={circle, inner sep=0pt, minimum size=3mm}]
			
			\node[roundnode, fill=black] (I1) at (2,1) {};
			\node[roundnode, fill=black] (F1) at (2,4) {};
			\node[roundnode, fill=black] (F2) at (2,3) {};
			\node[roundnode, fill=black] (F3) at (2,2) {};
			\node[roundnode, fill=black] (J1) at (4,0) {};    
			\node[roundnode, fill=black] (G6) at (4,1) {};  
			\node[roundnode, fill=black] (G4) at (3,2) {};
			\node[roundnode, fill=black] (G5) at (5,2) {};
			\node[roundnode, fill=black] (G2) at (3,3) {};
			\node[roundnode, fill=black] (G3) at (5,3) {};
			\node[roundnode, fill=black] (G1) at (4,4) {};
			\node[roundnode, fill=black] (B3) at (3,5) {};
			\node[roundnode, fill=black] (B2) at (3,6) {};
			\node[roundnode, fill=black] (B1) at (3,7) {};
			\node[roundnode, fill=black] (D1) at (7,4) {};
			\node[roundnode, fill=black] (D2) at (6,4) {};
			\node[roundnode, fill=black] (D3) at (8,4) {};
			\node[roundnode, fill=black] (C4) at (7,5) {};
			\node[roundnode, fill=black] (C3) at (7,6) {}; 
			\node[roundnode, fill=black] (C1) at (6,7) {};
			\node[roundnode, fill=black] (C2) at (8,7) {};
			\node[roundnode, fill=black] (A5) at (5,8) {};
			\node[roundnode, fill=black] (A3) at (4,9) {};
			\node[roundnode, fill=black] (A4) at (6,9) {};
			\node[roundnode, fill=black] (A2) at (5,10) {};
			\node[roundnode, fill=black] (A1) at (7,10) {}; 
			
			\path (A1) edge [ultra thick] (A4);
			\path (A2) edge [ultra thick] (A3);
			\path (A2) edge [ultra thick] (A4);
			\path (A3) edge [ultra thick] (A5);
			\path (A4) edge [ultra thick] (A5);
			\path (A5) edge [ultra thick] (B1);
			\path (A5) edge [ultra thick] (C1);
			\path (A5) edge [ultra thick] (C2);
			\path (C1) edge [ultra thick] (C3);
			\path (C2) edge [ultra thick] (C3);
			\path (C3) edge [ultra thick] (C4);
			\path (C4) edge [ultra thick] (D1);
			\path (C4) edge [ultra thick] (D2);
			\path (C4) edge [ultra thick] (D3);
			\path (B1) edge [ultra thick] (B2);
			\path (B2) edge [ultra thick] (B3);
			\path (B3) edge [ultra thick] (F1);
			\path (B3) edge [ultra thick] (G1);
			\path (F1) edge [ultra thick] (F2);
			\path (F2) edge [ultra thick] (F3);
			\path (F3) edge [ultra thick] (I1);
			\path (G1) edge [ultra thick] (G2);
			\path (G1) edge [ultra thick] (G3);
			\path (G2) edge [ultra thick] (G4);
			\path (G2) edge [ultra thick] (G5);
			\path (G3) edge [ultra thick] (G4);
			\path (G3) edge [ultra thick] (G5);
			\path (G5) edge [ultra thick] (G6);
			\path (G4) edge [ultra thick] (G6);
			\path (G6) edge [ultra thick] (J1);

			\node[roundnode, fill=black, label=right:{$ x_0=\ell $}] (J1) at (4,0) {};

			\node[roundnode, fill=red,label=left:{$ x_3 $}] (A5) at (5,8) {};
			\node[roundnode, fill=red,label=right:{$ x_2 $}] (B3) at (3,5) {};

			\draw[dashed,thick] (1.45,0.45) -- (8.25,0.45) -- (8.25,7.35) -- (1.45,7.35) -- (1.45,0.45);
			\draw[dashed,thick,color=purple] (1.55,0.55) -- (5.45,0.55) -- (5.45,7.25) -- (1.55,7.25) -- (1.55,0.55);

			\draw[dashed,thick,color=blue] (2.75,0.75) -- (2.75,4.25) -- (5.25,4.25) -- (5.25,0.75) -- (2.75,0.75);
			\draw[dashed,thick,color=green] (1.65,0.65) -- (5.35,0.65) -- (5.35,4.35) -- (1.65,4.35) -- (1.65,0.65);
			
			\node[roundnode, fill=black,label=right:{$ x_1 $}] (G6) at (4,1) {};  
			
			\node[roundnode, fill=red,label=left:{$ u_{\ell,3} $}] (R1) at (14,6) {};
			\node[roundnode, fill={red},label=left:{$ u_{\ell,2} $}] (R2) at (12,5) {};
			\node[roundnode, fill={black}] (R3) at (16,5) {};
			\node[roundnode, fill={black}] (R4) at (11,4) {};
			\node[roundnode, fill={black},label=right:{$ u_{\ell,1} $}] (R5) at (13,4) {};    
			\node[roundnode, fill={black}] (R6) at (15,4) {};  
			\node[roundnode, fill={black}] (R7) at (16,4) {};
			\node[roundnode, fill={black}] (R8) at (17,4) {};
			\node[roundnode, fill={black}] (R9) at (11,3) {};
			\node[roundnode, fill={black},label=right:{$ u_{\ell,0} $}] (R10) at (13,3) {};
			
			\path (R1) edge [ultra thick] (R2);
			\path (R1) edge [ultra thick] (R3);
			\path (R2) edge [ultra thick] (R4);
			\path (R2) edge [ultra thick] (R5);
			\path (R3) edge [ultra thick] (R6);
			\path (R3) edge [ultra thick] (R7);
			\path (R3) edge [ultra thick] (R8);
			\path (R4) edge [ultra thick] (R9);
			\path (R5) edge [ultra thick] (R10);
			
			\node (P) at (5,-1) {$ P $};
			\node (Q) at (14,2) {$ Q $};

	\end{tikzpicture}		}
	\caption{An illustration of the notation defined above, where $ P $ is the inflated rooted tree from \Cref{fig:inflated rooted tree}. The black and green boxes denote $ X_3 $ and $ X_2 $, respectively, while the red and blue boxes denote $ A_3 $ and $ A_2 $, respectively.}
\end{figure}
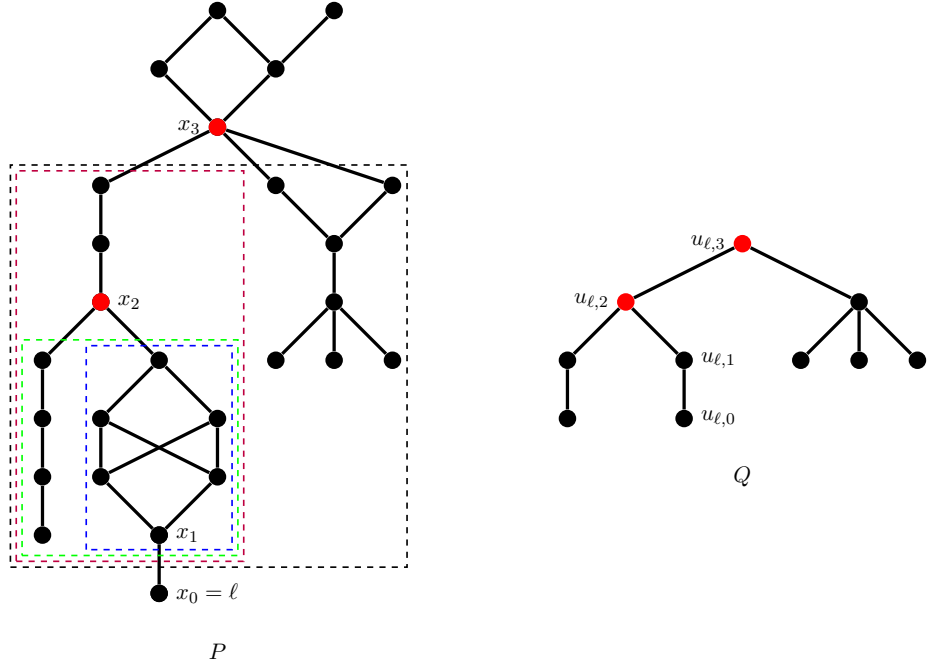

\begin{lemma}\label{A2}
	Fix $ r\in\{1,2\} $. Let $ P $, $ x_0 $, $ \ell $, and the $ A_j $'s, $ A_j' $'s, and $ X_j $'s be defined as above. If $ x_0 $ is minimal, fix $ a\in\{n-1,n\} $. Otherwise fix $ a=n $. Set $ L(x_0)=a $. Then $ L_{n-3}\inverse(r)\in A_2' $ if and only if $ \tilde{L}_{n-3}\inverse(r)\in A_2 $.
\end{lemma}
\begin{proof}
	Suppose $ x_0 $ is minimal. Then $ L_{n-3}(x_0)=L(x_0)-n+3\in\{2,3\} $, and $ x_0 $ is never in the promotion chain for the first $ n-3 $ promotions. The lemma follows immediately from applying \Cref{standardization} to $ P $, $ L $, and $ x_0 $.
	
	Suppose $ x_0 $ covers a unique minimal element $ \ell $ and $ L(x_0)=n $. Also assume that $ L_{n-3}\inverse(r)\in A_2' $. We claim that this implies $ x_0 $ is not in the promotion chain for the first $ n-3 $ promotions. Suppose to the contrary that $ x_0 $ is in the $ \gamma $th promotion chain for some $ 0\leq\gamma\leq n-4 $. This forces $ L_\gamma(\ell)=1 $ and implies that $ x_0 $ is the $ L_\gamma $-successor of $ \ell $. Note that $ L_{n-3}\inverse(r)\in A_2' $ implies that $ L_{\alpha}\inverse(r+n-3-\alpha) $ is comparable to $ x_0 $ for all $ 0\leq\alpha\leq n-3 $. In particular, $ L_{\gamma}\inverse(r+n-3-\gamma) $ must be above $ x_0 $, since $ x_0 $ is above only $ \ell $ and $ L_{\gamma}(\ell)=1 $. It follows that $ x_0 $ cannot be the $ L_{\gamma} $-successor of $ \ell $, since $ r+n-3-\gamma<n-\gamma=L_{\gamma}(x_0) $. This is a contradiction, so $ x_0 $ is not in the promotion chains of $ L,\ldots,L_{n-4} $. Hence, we may apply \Cref{standardization}, and it follows that $ \tilde{L}_{n-3}\inverse(r)\in A_2 $. 
	
	For the converse, assume that $ L_{n-3}\inverse(r)\not\in A_2' $. We have two cases: (1) $ x_0 $ is not in the promotion chains of $ L,\ldots,L_{n-4} $; (2) $ x_0 $ is in the promotion chain of $ L_{\gamma} $ for some $ \gamma\in\{0,\ldots,n-4\} $. For case (1), we simply apply \Cref{standardization} and are done. 
	
	For case (2), we note that for all $ 0\leq\alpha\leq\gamma $, \Cref{standardization} implies that $ \st(L_{\alpha}|_{\tilde{P}})=\tilde{L}_{\alpha} $. In particular, we have $ \st(L_{\gamma}|_{\tilde{P}})=\tilde{L}_{\gamma} $. Since we are assuming that $ x_0 $ is in the promotion chain of $ L_{\gamma} $, it follows that $ L_{\gamma}(\ell)=1 $ and that $ x_0 $ is the $ L_{\gamma} $-successor of $ \ell $. Hence, with respect to $ L_{\gamma} $, there are no elements of $ P $ above $ x_0 $ with label smaller than $ n-\gamma $. In particular, $ L_{\gamma}\inverse(r+n-3-\gamma) $ is not comparable to $ x_0 $ or $ \ell $. Since $ \st(L_{\gamma}|_{\tilde{P}})=\tilde{L}_{\gamma} $, it follows that $ \tilde{L}_{\gamma}\inverse(r+n-3-\gamma) $ is not comparable to $ \ell $ in $ \tilde{P} $. Therefore, $ \tilde{L}_{n-3}\inverse(r)\not\in A_2 $, as desired. 
\end{proof}
The next step is using this machinery to compute the conditional probabilities $ \Prob(E_{\ell, j}|E_{\ell, j+1}) $ as well as $ \Prob(E_{\ell,\omega(\ell)}) $.
\begin{lemma}\label{conditional helper}
	Let $ P $, $ x_0 $, $ \ell $, $ K $, and the $ E_{\ell,j} $'s, $ A_j $'s, $ X_j $'s, $ b_{\ell,j} $'s, and $ c_{\ell,j} $'s be defined as above. Let $ j\in\{2,\ldots,\omega(\ell)-1\} $, and fix any injective map \[M_j:\tilde{P}\setminus X_j\to[n-1]\] such that every labeling $ \tilde{L}:\tilde{P}\to[n-1] $ extending $ M_j $ has the property that $ E_{\ell,j+1} $ occurs. Consider the uniform distribution on such labelings $ \tilde{L} $. Then \[\Prob(E_{\ell,j}\,|\,E_{\ell,j+1})=\prod_{t=1}^{|K|}\frac{|A_j|-t+1}{|X_j|-t+1}=\prod_{t=1}^{|K|}\frac{b_{\ell,j}-t}{c_{\ell,j}-t}.\]
\end{lemma}
\begin{proof}
	Recall that we may always assume $ Q $ has more than one element and thus that $ \omega(\ell)\geq1 $. Also, by \Cref{reduced}, we have that $ |A_2|\geq2 $. 
	
	By hypothesis, every labeling $ \tilde{L}:\tilde{P}\to[n-1] $ extending $ M_j $ has the property that $ E_{\ell,j+1} $ occurs. Recall that this implies $ K\subset\tilde{L}_{n-3}(A_{j+1}) $ and hence that $ K\subset\tilde{L}_{n-3}(X_j) $, since $ \{\tilde{L}_{n-3}\inverse(1),\tilde{L}_{n-3}\inverse(2)\} $ forms a lower order ideal of $ \tilde{P} $. By \Cref{label distribution}, $ \tilde{L}_{\gamma} $ depends only on $ \tilde{L}|_{\tilde{P}\setminus X_j} $. Hence, $ \Prob(E_{\ell,j}\,|\,E_{\ell,j+1}) $ depends only on $ \tilde{L}|_{\tilde{P}\setminus X_j} $.  Apply the Probability Lemma (\Cref{probability lemma}) with $ N=n-1 $, $ x=x_j $, $ X=X_j $, $ A=A_j $, $ M=M_j $, $ k=n-3 $, and $ \{n_1,\ldots,n_d\}=K $. This tells us that \[\Prob(E_{\ell,j}\,|\,E_{\ell,j+1})=\prod_{t=1}^{|K|}\frac{|A_j|-t+1}{|X_j|-t+1}.\] The lemma follows.  
\end{proof}
\begin{lemma}\label{helper 1}
	With notation as in the previous lemma, fix any injective map \[M_{\omega(\ell)}:\tilde{P}\setminus X_{\omega(\ell)}\to[n-1],\] and consider the uniform distribution on the labelings $ \tilde{L}:\tilde{P}\to[n-1] $ extending $ M_{\omega(\ell)} $. Then \[\Prob(E_{\ell,\omega(\ell)})=\prod_{t=1}^{|K|}\frac{|A_{\omega(\ell)}|-t+1}{X_{\omega(\ell)}-t+1}=\prod_{t=1}^{|K|}\frac{b_{\ell,\omega(\ell)}-t}{c_{\ell,\omega(\ell)}-t}.\]
\end{lemma}
\begin{proof}
	We split into cases based on whether or not $ P $ has a unique minimal element. Suppose $ P $ has a unique minimal element. Then $ A_j=X_j $, and, consequentially, $ b_{\ell,j}=c_{\ell,j} $. Hence, it suffices to show that the probability in question is 1. Since $ \{\tilde{L}_{n-3}\inverse(1),\tilde{L}_{n-3}\inverse(2)\} $ forms a lower order ideal of size 2, it is not difficult to see that when $ P $ has a unique minimal element, $ \{\tilde{L}_{n-3}\inverse(1),\tilde{L}_{n-3}\inverse(2)\}\subset A_2 $. Hence, $ K\subset\tilde{L}_{n-3}(A_2) $. Because $ A_2\subset A_j $, the probability in question is 1, as desired. 
	
	Suppose $ P $ does not have a unique minimal element. The argument is identical to that in \Cref{conditional helper} as long as we show that for any such labeling $ \tilde{L} $, $ K\subset\tilde{L}_{n-3}(X_{\omega(\ell)}) $. This simply follows from recalling that $ \{\tilde{L}_{n-3}\inverse(1),\tilde{L}_{n-3}\inverse(2)\} $ forms a lower order ideal of size 2, because $ |\tilde{P}|=n-1 $. Since $ P $ does not have a unique minimal element, $ x_{\omega(\ell)} $ is greater than at least two elements, implying $ \{\tilde{L}_{n-3}\inverse(1),\tilde{L}_{n-3}\inverse(2)\}\subset X_{\omega(\ell)}=\{y\in\tilde{P}\;|\;y<_{\tilde{P}}x_{\omega(\ell)}\} $. 
\end{proof}
The previous two lemmas only give us information about $ \tilde{P} $. In the following, we use \Cref{A2} to translate these results into information about $ P $. 
\begin{lemma}\label{helper}
	Let $ \vphi:P\to Q $ be an inflation of a rooted tree poset with deflated leaves, and let $ n $ be the number of elements in $ P $. Let $ x_0 $ cover a unique minimal element or be minimal itself. If $ x_0 $ is a minimal element, let $ \ell=x_0 $; otherwise let $ \ell $ be the minimal element covered by $ x_0 $. If $ x_0 $ is minimal, fix $ a\in\{n-1,n\} $. Otherwise fix $ a=n $. Let the $ A_j $'s, $ X_j $'s, $ A_j' $'s, $ E_{\ell,j} $'s, $ b_{\ell,j} $'s, and $ c_{\ell,j} $'s be defined as above. Let $ K$ be some nonempty subset of $\{1,2\} $. Suppose $ L:P\to[n] $ is a labeling chosen uniformly at random among the $ (n-1)! $ labelings with $ L(x_0)=a $. Then the probability that every label in $ K $ is in $ L_{n-3}(A_2') $ is \[\prod_{j=2}^{\omega(\ell)}\prod_{t=1}^{|K|}\frac{b_{\ell,j}-t}{c_{\ell,j}-t}.\]
\end{lemma}
\begin{proof}
	Note that $ L $ induces the uniform distribution on labelings $ \tilde{L}:\tilde{P}\to[n-1] $. By \Cref{A2}, $ K\subset L_{n-3}(A_2') $ if and only if $ K\subset \tilde{L}_{n-3}(A_2) $. Thus, we would like to calculate \[\Prob(E_{\ell,2})=\Prob(E_{\ell,\omega(\ell)})\Prob(E_{\ell,\omega(\ell)-1}\,|\,E_{\ell,\omega(\ell)})\cdots\Prob(E_{\ell,2}\,|\,E_{\ell,3}).\] The result follows from applying \Cref{helper 1} and \Cref{conditional helper}.
\end{proof}
The following three lemmas are applications of \Cref{helper} to the configurations of interest for the proof of \Cref{quasi tangled enumeration}:
\begin{lemma}\label{L^-1(n-1) minimal}
	With notation as in \Cref{quasi tangled enumeration}, the number of labelings $ L $ of $ P $ with $ L\inverse(n-1)\not\in S $, $ L\inverse(n-1) $ minimal, and $ L_{n-3}\inverse(2)<_PL_{n-3}\inverse(1) $ is \[(n-1)!\left(\sum_{\ell\in T}\prod_{j=2}^{\omega(\ell)}\frac{b_{\ell,j}-1}{c_{\ell,j}-1}+\sum_{\ell\in R}\prod_{j=2}^{\omega(\ell)}\frac{b_{\ell,j}-1}{c_{\ell,j}-1}\right).\]
\end{lemma}
\begin{proof}
	Begin by noting that $ L_{n-3}\inverse(2)=L\inverse(n-1) $. Because $ L\inverse(n-1) $ is minimal, \Cref{lower order ideal of size 3} gives us that $ L\inverse(n-1)\in T $, $ L\inverse(n-1)\in R $, or $ L\inverse(n-1)\in S $. 
	
	Case (1): Assume $ L\inverse(n-1)\in T $. We would like to compute the probability that $ L_{n-3}\inverse(1)>_P L\inverse(n-1)=L_{n-3}\inverse(2) $. Note that this event occurs if and only if $ L_{n-3}\inverse(1)\in A_2' $. Applying \Cref{helper} with $ x_0=\ell=L\inverse(n-1) $, we see that this probability is just \[\prod_{j=2}^{\omega(\ell)}\frac{b_{\ell,j}-1}{c_{\ell,j}-1}.\] Summing over the minimal elements in $ T $, we get the summation corresponding to $ T $ in the formula. 
	
	Case (2): An analogous argument works for $ L\inverse(n-1)\in R $. The lemma follows.
\end{proof}
\begin{lemma}\label{L^-1(n) minimal}
	With notation as in \Cref{quasi tangled enumeration}, the number of labelings $ L $ of $ P $ with $ L\inverse(n) $ minimal and $ L_{n-3}\inverse(3)<_PL_{n-3}\inverse(1) $ or $ L_{n-3}\inverse(3)<_PL_{n-3}\inverse(2) $ is \begin{equation*}
		\resizebox{\textwidth}{!} 
		{
			$(n-1)!\left(2\displaystyle\sum_{\ell\in T}\prod_{j=2}^{\omega(\ell)}\frac{b_{\ell,j}-1}{c_{\ell,j}-1}+2\sum_{\ell\in R}\prod_{j=2}^{\omega(\ell)}\frac{b_{\ell,j}-1}{c_{\ell,j}-1}\\
			+\sum_{\ell\in S}\prod_{j=2}^{\omega(\ell)}\frac{(b_{\ell,j}-1)(b_{\ell,j}-2)}{(c_{\ell,j}-1)(c_{\ell,j}-2)}-\sum_{\ell\in T}\prod_{j=2}^{\omega(\ell)}\frac{(b_{\ell,j}-1)(b_{\ell,j}-2)}{(c_{\ell,j}-1)(c_{\ell,j}-2)}\right).$
		}
	\end{equation*}
\end{lemma}
\begin{proof}
	Recall that \Cref{lower order ideal of size 3} implies that $ L\inverse(n) $ is in either $ R $, $ S $, or $ T $. When $ L\inverse(n) $ is in $ R $ or $ T $, the process of counting the number of such labelings (where $ L\inverse(n) $ is minimal and $ L_{n-3}\inverse(3)<_PL_{n-3}\inverse(1) $ or $ L_{n-3}\inverse(3)<_PL_{n-3}\inverse(2) $) is nearly identical to the one used in the proof of \Cref{L^-1(n-1) minimal} (just apply \Cref{helper}). However, if $ L\inverse(n)\in T $, it is possible that $ L_{n-3}\inverse(3)<_PL_{n-3}\inverse(1) $ \emph{and} $ L_{n-3}\inverse(3)<_PL_{n-3}\inverse(2) $; we are twice-counting such labelings. To count the labelings where $ L\inverse(n)\in T $, $ L_{n-3}\inverse(3)<_PL_{n-3}\inverse(1) $, and $ L_{n-3}\inverse(3)<_PL_{n-3}\inverse(2) $, we apply \Cref{helper}. Thus, for $ L\inverse(n)\in R\cup T $, there are \[(n-1)!\left(2\sum_{\ell\in T}\prod_{j=2}^{\omega(\ell)}\frac{b_{\ell,j}-1}{c_{\ell,j}-1}+2\sum_{\ell\in R}\prod_{j=2}^{\omega(\ell)}\frac{b_{\ell,j}-1}{c_{\ell,j}-1}-\sum_{\ell\in T}\prod_{j=2}^{\omega(\ell)}\frac{(b_{\ell,j}-1)(b_{\ell,j}-2)}{(c_{\ell,j}-1)(c_{\ell,j}-2)}\right)\] labelings where $ L_{n-3}\inverse(3)<_PL_{n-3}\inverse(1) $ or $ L_{n-3}\inverse(3)<_PL_{n-3}\inverse(2) $. The term being subtracted in the above expression is the number of labelings with  $ L\inverse(n)\in T $, $ L_{n-3}\inverse(3)<_PL_{n-3}\inverse(1) $, and $ L_{n-3}\inverse(3)<_PL_{n-3}\inverse(2) $.
	
	However, the process changes when $ L\inverse(n)\in S $. Let $ m $ be the other element covered by the parent of $ L\inverse(n) $; let $ Y $ be the lower order ideal of size 3 consisting of $ L\inverse(n)$, $ m $, and their parent. Note that we must have $ Y=\{L_{n-3}\inverse(1),L_{n-3}\inverse(2),L_{n-3}\inverse(3)\} $. Setting $ L\inverse(n)=x_0=\ell $, with notation as in \Cref{helper}, we have that $ Y=\{L_{n-3}\inverse(1),L_{n-3}\inverse(2),L_{n-3}\inverse(3)\} $ if and only if $ L_{n-3}\inverse(1) $ and $ L_{n-3}\inverse(2) $ are in $ A_2' $. By \Cref{helper}, the probability that both $ L_{n-3}\inverse(1) $ and $ L_{n-3}\inverse(2) $ are in $ A_2' $ is \[\prod_{j=2}^{\omega(\ell)}\frac{(b_{\ell,j}-1)(b_{\ell,j}-2)}{(c_{\ell,j}-1)(c_{\ell,j}-2)}.\] The formula follows from summing over all elements in $ S $.
\end{proof}
\begin{lemma}\label{L^-1(n) covers min element}
	With notation as in \Cref{quasi tangled enumeration}, the number of labelings $ L $ where $ L\inverse(n) $ covers a minimal element and $ L_{n-3}\inverse(1)<_PL\inverse(n)<_PL_{n-3}\inverse(2) $ or $ L_{n-3}\inverse(2)<_PL\inverse(n)<_PL_{n-3}\inverse(1) $ is \[(n-1)!\left(\sum_{m\in T}\prod_{j=2}^{\omega(m)}\frac{(b_{m,j}-1)(b_{m,j}-2)}{(c_{m,j}-1)(c_{m,j}-2)}\right).\]
\end{lemma}
\begin{proof}
	Let $ x_0=L\inverse(n) $, and let notation be as in \Cref{helper} so that $ L\inverse(n) $ covers some minimal element $ \ell$. Note that $ \ell\in T $ since $ L_{n-3}\inverse(1), $ $ L_{n-3}\inverse(2) $, and $ L_{n-3}\inverse(3)=L\inverse(n) $ form a lower order ideal. Moreover, note that $ L_{n-3}\inverse(1) $ and $ L_{n-3}\inverse(2) $ are comparable to $ x_0 $ if and only if they are in $ A_2' $. Hence, we may apply \Cref{helper} to see that the probability both $ L_{n-3}\inverse(1) $ and $ L_{n-3}\inverse(2) $ are comparable to $ x_0 $ is \[\prod_{j=2}^{\omega(m)}\frac{(b_{m,j}-1)(b_{m,j}-2)}{(c_{m,j}-1)(c_{m,j}-2)}.\] Summing over all the elements in $ T $ will imply the lemma. 
\end{proof}

\begin{proof}[Proof of \Cref{quasi tangled enumeration}]
	We begin by counting the number of tangled labelings of $ P $. By Lemma 3.8 in \cite{DK20}, if $ L $ is a tangled labeling of an $ n $-element poset, then $ L\inverse(n) $ is minimal. Moreover, a labeling is tangled if and only if $ L_{n-2}\inverse(1)>_PL_{n-2}\inverse(2) $. If $ L $ is tangled, then it follows that $ L\inverse(n)\in R\cup T $, since $ \{L_{n-2}\inverse(1),L_{n-2}\inverse(2)\} $ forms a lower order ideal and since $ L_{n-2}\inverse(1)>_PL_{n-2}\inverse(2) $. Applying \Cref{helper}, we see that there are \begin{equation}\label{eq: tangled}
		(n-1)!\left(\sum_{\ell\in T}\prod_{j=2}^{\omega(\ell)}\frac{b_{\ell,j}-1}{c_{\ell,j}-1}+\sum_{\ell\in R}\prod_{j=2}^{\omega(\ell)}\frac{b_{\ell,j}-1}{c_{\ell,j}-1}\right)
	\end{equation} tangled labelings. 
	
	Now, we enumerate the labelings $ L $ such that $ L_{n-3}\not\in\call(P) $. \Cref{L^-1(n) position} tells us that if $ L_{n-3}\not\in\call(P) $, then either $ L\inverse(n-1) $ is minimal, $ L\inverse(n) $ is minimal, or $ L\inverse(n) $ covers a minimal element. Moreover, we know that $ Y=\{L_{n-3}\inverse(1),L_{n-3}\inverse(2),L_{n-3}\inverse(3)\} $ forms a lower order ideal of size 3, and $ L_{n-3} $ restricted to $ Y $ is not a linear extension. We condition on the three cases given by \Cref{L^-1(n) position}.
	
	Case (1): We count the labelings $ L $ such that $ L\inverse(n-1) $ is minimal, $ L\inverse(n-1)\not\in S $, and $ L_{n-3}\inverse(1)>_PL_{n-3}\inverse(2) $. By \Cref{L^-1(n-1) minimal}, there are \begin{equation}\label{eq: L(n-1)}
		(n-1)!\left(\sum_{\ell\in T}\prod_{j=2}^{\omega(\ell)}\frac{b_{\ell,j}-1}{c_{\ell,j}-1}+\sum_{\ell\in R}\prod_{j=2}^{\omega(\ell)}\frac{b_{\ell,j}-1}{c_{\ell,j}-1}\right)
	\end{equation} such labelings. If $ L\inverse(n-1)\in S $, because $ \{L_{n-3}\inverse(1),L_{n-3}\inverse(2),L_{n-3}\inverse(3)\} $ is a lower order ideal, it follows from the definition of $ S $ that $ L_{n-3}\inverse(1) $ must be the unique parent of both $ L_{n-3}\inverse(2) $ and $ L_{n-3}\inverse(3) $. Now, repeatedly applying \Cref{inversion} to $ L_{n-3}\inverse(1) $ and $ L_{n-3}\inverse(3) $ tells us the position of $ L\inverse(n) $, namely that $ L\inverse(n)=L_{n-3}\inverse(3)\in S $. Thus, this subcase can be excluded and will be addressed in Case (2) when we assume $ L\inverse(n) $ is minimal. 
	
	Case (2): We count of labelings $ L $ such that $ L\inverse(n) $ is minimal and $ L_{n-3}\inverse(3)<_PL_{n-3}\inverse(1) $ or $ L_{n-3}\inverse(3)<_PL_{n-3}\inverse(2) $. By \Cref{L^-1(n) minimal}, there are \begin{equation}\label{eq: L(n)}
		\resizebox{.94 \textwidth}{!} 
		{
			$(n-1)!\left(2\displaystyle\sum_{\ell\in T}\prod_{j=2}^{\omega(\ell)}\frac{b_{\ell,j}-1}{c_{\ell,j}-1}+2\sum_{\ell\in R}\prod_{j=2}^{\omega(\ell)}\frac{b_{\ell,j}-1}{c_{\ell,j}-1}\\
			+\sum_{\ell\in S}\prod_{j=2}^{\omega(\ell)}\frac{(b_{\ell,j}-1)(b_{\ell,j}-2)}{(c_{\ell,j}-1)(c_{\ell,j}-2)}-\sum_{\ell\in T}\prod_{j=2}^{\omega(\ell)}\frac{(b_{\ell,j}-1)(b_{\ell,j}-2)}{(c_{\ell,j}-1)(c_{\ell,j}-2)}\right)$
		}
	\end{equation} such labelings.
	
	Case (3): We count the labelings $ L $ such that $ L\inverse(n) $ covers a minimal element $ \ell $, $ L_{n-3}\inverse(2)>_PL_{n-3}\inverse(3) $ or $ L_{n-3}\inverse(1)>_PL_{n-3}\inverse(3) $, and $ L(\ell)\neq n-1 $. (The case where $ L(\ell)=n-1 $ and $ L_{n-3}\inverse(1)>_PL_{n-3}\inverse(3)>_PL_{n-3}\inverse(2) $ was counted in Case (1).) We first count the labelings $ L $ such that  $ L_{n-3}\inverse(2)>_PL_{n-3}\inverse(3) $ or $ L_{n-3}\inverse(1)>_PL_{n-3}\inverse(3) $. Since $ Y $ is a lower order ideal of size 3, it follows that $ \ell\in Y $. Thus, we may assume that $ \ell\in T $. Hence, it is sufficient to count the labelings $ L $ such that $ L\inverse(n) $ covers some $ \ell\in T $ and $ L_{n-3}\inverse(1)<_PL\inverse(n)<_PL_{n-3}\inverse(2) $ or $ L_{n-3}\inverse(2)<_PL\inverse(n)<_PL_{n-3}\inverse(1) $. We have already done this---the number of such labelings is given in \Cref{L^-1(n) covers min element}. Note that each such labeling is indeed quasi-tangled. To account for the condition $ L(\ell)\neq n-1 $, we enumerate the labelings with $ L\inverse(n-1)\in T $, $ L\inverse(n-1)\lessdot_PL\inverse(n) $, and $ L_{n-3}\inverse(1)>_P L_{n-3}\inverse(3) $. An adaptation of \Cref{helper} allows us to enumerate these labelings, the number of which is given by the term being subtracted in the following expression: \begin{equation}\label{eq: L(n) covers}
		(n-1)!\left(\sum_{m\in T}\prod_{j=2}^{\omega(m)}\frac{(b_{m,j}-1)(b_{m,j}-2)}{(c_{m,j}-1)(c_{m,j}-2)}-\frac{1}{n-1}\sum_{m\in T}\prod_{j=2}^{\omega(m)}\frac{(b_{m,j}-2)}{(c_{m,j}-2)}\right).
	\end{equation} Note that the above enumerates the labelings $ L $ such that $ L\inverse(n) $ covers a minimal element $ \ell $, $ L_{n-3}\inverse(2)>_PL_{n-3}\inverse(3) $ or $ L_{n-3}\inverse(1)>_PL_{n-3}\inverse(3) $, and $ L(\ell)\neq n-1 $.
	
	Summing \eqref{eq: L(n-1)}, \eqref{eq: L(n)}, and \eqref{eq: L(n) covers} and subtracting \eqref{eq: tangled} gives that the number of quasi-tangled labelings of $ P $ is given by \begin{equation*}\resizebox{\textwidth}{!} 
		{
			$(n-1)!\left(2\displaystyle\sum_{\ell\in T}\prod_{j=2}^{\omega(\ell)}\frac{b_{\ell,j}-1}{c_{\ell,j}-1}-\frac{1}{n-1}\sum_{\ell\in T}\prod_{j=2}^{\omega(\ell)}\frac{b_{\ell,j}-2}{c_{\ell,j}-2}+2\sum_{\ell\in R}\prod_{j=2}^{\omega(\ell)}\frac{b_{\ell,j}-1}{c_{\ell,j}-1}+\sum_{\ell\in S}\prod_{j=2}^{\omega(\ell)}\frac{(b_{\ell,j}-1)(b_{\ell,j}-2)}{(c_{\ell,j}-1)(c_{\ell,j}-2)}\right)$
		}
	\end{equation*} as desired. 
\end{proof}

\section{Enumerating Labelings of Rooted Trees with Sorting Time $ n-1-k $}\label{sorting time n-k-1}
In light of \Cref{L^-1(n) position} and the Probability Lemma (\Cref{probability lemma}), it is natural to ask if the methods used in \Cref{Quasi-Tangled} can be extended to enumerate the labelings of an $ n $-element poset $ P $ with sorting time $ n-1-k $. In the following, we give an algorithmic approach for doing so when $ P $ is a rooted tree poset. While, theoretically, this approach could yield a general formula for the labelings with sorting time $ n-1-k $, any such formula would be much too complicated to be practical. Instead, for a fixed $ k $, we offer an algorithmic approach to enumerating the labelings with sorting time $ n-1-k $. Using this method, it would be possible to write a computer program that computes the number of such labelings for a fixed poset.

In order to do so, we first note that \Cref{L^-1(n) position} generalizes. In particular, for a fixed $ k $, one can prove that if $ L $ has sorting time $ n-1-k $, then one of the following holds:
\begin{itemize}
	\item $ L\inverse(n-k) $ is minimal;
	\item $ L\inverse(n-k+1) $ is minimal or covers a minimal element;
	\item $ L\inverse(n-k+2) $ is minimal, covers a minimal element, or is greater than exactly 2 other elements;
	\subitem\vdots
	\item $ L\inverse(n) $ is greater than at most $ k $ other elements. 
\end{itemize}

\begin{Algorithm}
	Let $ P $ be an $ n $-element rooted tree poset. Then we may enumerate the labelings $ L $ of $ P $ with sorting time $ n-1-k $ in the following way:
	\begin{enumerate}
		\item List the possible lower order ideals of size $ k $ other than antichains appearing in $ P $.
		\item For each such lower order ideal occurring in $ P $, use (1) and the Probability Lemma to count the labelings with sorting time $ n-1-k $. This will involve lots of casework based on the positions of $ L\inverse(n-k),\ldots,L\inverse(n) $ and lower order ideals of size $ k $ occurring in $ P $. The proof of \Cref{quasi tangled enumeration} illustrates this casework for the case $ k=1 $ in full generality.  
	\end{enumerate}
\end{Algorithm}

\section{Tangled Labelings of Posets}\label{Tangled}
In \cite{DK20}, Defant and Kravitz conjectured that any $ n $-element poset has at most $ (n-1)! $ tangled labelings (see \Cref{conj: (n-1)!}). This is not obvious even for classes of posets for which we can explicitly enumerate the tangled labelings (e.g., \Cref{tangled enumeration}).

We can prove the conjecture for inflated rooted forests. The following results from \cite{DK20} will be useful:
\begin{corollary}[\cite{DK20}, Corollary 3.7]\label{cor:tangled unique minimal element}
	Let $ P $ be an $ n $-element poset with $ r $ connected components, each having a unique minimal element. Then the number of tangled labelings of $ P $ is \[(n-r)(n-2)!.\]
\end{corollary}
\begin{thm}[\cite{DK20}, Theorem 3.4]
	Let $ P $ be an $ n $-element poset with connected components $ P_1,\ldots,P_r $. Let $ n_i=|P_i| $, and let $ t_i $ denote the number of tangled labelings of $ P_i $. The number of tangled labelings of $ P $ is \[(n-2)!\sum_{i=1}^r\frac{t_i}{(n_i-2)!}.\]
\end{thm}

\begin{lemma}\label{reduction to connected case}
	Let $ P $, $ P_i $, $ n_i $, and $ t_i $ be defined as in the above for $ i=1,\ldots,r $. If there are at most $ (n_i-1)! $ tangled labelings of each $ P_i $, then there are at most $ (n-r)(n-2)! $ tangled labelings of $ P $.
\end{lemma}
\begin{proof}
	Substituting $ (n_i-1)!\geq t_i $ into \[(n-2)!\sum_{i=1}^r\frac{t_i}{(n_i-2)!}\leq(n-2)!\sum_{i=1}^r(n_i-1)=(n-2)!(n-r)\] gives the bound.
\end{proof}
\begin{thm}\label{(n-1)!}
	Let $ P $ be an $ n $-element inflated rooted forest poset. Then $ P $ has at most $ (n-1)! $ tangled labelings. Equality holds if and only if $ P $ has a unique minimal element.
\end{thm}
\begin{proof}
	By \Cref{reduction to connected case}, it suffices to prove this for $ P $ an inflated rooted tree, where $ Q $ is the rooted tree and $ \varphi:P\to Q $ the inflation map. Assume without loss of generality that $ Q $ is reduced. We know that the number of tangled labelings of $ P $ is \begin{equation}\label{tangled}
		(n-1)!\sum_{i=1}^{s}\prod_{j=1}^{\omega(i)}\frac{b_{i,j}-1}{c_{i,j}-1}.
	\end{equation} Let $ \ell_1,\ldots,\ell_s $ denote the leaves of $ Q $, and let $ m_1,\ldots,m_s $ be the unique minimal elements of $ \varphi\inverse(\ell_1),\ldots,\varphi\inverse(\ell_s) $, respectively. Suppose without loss of generality that $ \ell_{s-1} $ and $ \ell_s $ have the same parent in $ Q $ (such leaves exist because $ Q $ is assumed to be reduced). For all $ j\neq1 $, $ b_{s-1,j}=b_{s,j} $; for all $ j $, $ c_{s-1,j}=c_{s,j} $. Hence, we may rewrite \eqref{tangled} as \[(n-1)!\left(\sum_{i=1}^{s-2}\prod_{j=1}^{\omega(i)}\frac{b_{i,j}-1}{c_{i,j}-1}+\prod_{j=2}^{\omega(s-1)}\frac{b_{s-1,j}-1}{c_{s-1,j}-1}\left(\frac{b_{s-1,1}+b_{s,1}-2}{c_{s-1,1}-1}\right)\right).\] Now, let $ P' $ be the poset obtained from $ P $ by adding the additional relation $ m_{s-1}\lessdot_{P'}m_s $. Note that the resulting poset is still an inflated rooted tree poset and that $ P' $ is an inflation of $ Q' $, where $ Q' $ is the (reduced) rooted tree poset formed by setting $ \ell_s=\ell_{s-1} $ and reducing if necessary. Let $ \psi:P'\to Q' $ be the corresponding inflation map. Moreover, note that $ Q' $ has $ s-1 $ leaves. For each leaf $ \ell_1',\ldots,\ell_{s-1}' $ in $ Q' $, let $ u_{i,0}',\ldots,u_{i,\omega(i)}' $ denote the unique path from $ \ell_i' $ to the root of $ Q' $. Let \[b_{i,j}'=\sum_{v\leq_{Q'} u_{i,j-1}'}|\psi\inverse(v)|\qquad\mathrm{and}\qquad c_{i,j}'=\sum_{v<_{Q'} u_{\ell,j}'}|\psi\inverse(v)|.\] Note that for $ i\in\{1,\ldots,s-2\} $ and $ j\in\{1,\ldots,\omega(i)\} $, $ b_{i,j}=b_{i,j}' $ and $ c_{i,j}=c_{i,j}' $. Moreover, we also know that when $ i=s-1 $, $ b_{i,j}=b_{i,j}' $ and $ c_{i,j}=c_{i,j}' $ for $ j=2,\ldots,\omega(i) $. When $ j=1 $, we have $ b_{s-1,1}'=b_{s-1,1}+b_{s,1} $ and $ c_{i,j}'=c_{i,j} $. It follows that the number of tangled labelings of $ P' $ is \[(n-1)!\left(\sum_{i=1}^{s-2}\prod_{j=1}^{\omega(i)}\frac{b_{i,j}-1}{c_{i,j}-1}+\prod_{j=2}^{\omega(s-1)}\frac{b_{s-1,j}-1}{c_{s-1,j}-1}\left(\frac{b_{s-1,1}+b_{s,1}-1}{c_{s-1,1}-1}\right)\right).\] Note that $ P' $ has more tangled labelings than $ P $ and that $ P' $ has $ s-1 $ minimal elements. 
	
	The result follows from induction and \Cref{cor:tangled unique minimal element}.
\end{proof}

\section{Open Problems}\label{Open}
In light of \Cref{noninvertibility bound}, it is natural to ask for which rooted tree posets $ \deg(\partial) $ is biggest. The following is motivated by the fact that $ \partial $ is injective if and only if $ P $ is an antichain. Since the degree of noninvertibility of promotion is smallest when $ P $ is an antichain, it seems natural that $ \deg(\partial) $ would be largest when $ P $ is a chain. In Remark 2.4 of \cite{DK20}, it is shown that when $ P $ is an $ n $-element chain, $ \partial $ is dynamically equivalent to the bubble sort map from \cite{Knuth1973} (pages 106-110). Theorem 2.2 of \cite{DP20} tells us that when $ P $ is an $ n $-element chain, $ \deg(\partial)=(n+2)(n+1)/{6} $. \begin{conjecture}
	For any poset $ P $, \[\deg(\partial)\leq\frac{(n+2)(n+1)}{6}.\] In other words, the degree of noninvertibility of promotion is largest when $ P $ is a chain.
\end{conjecture}

\Cref{quasi tangled enumeration}, in conjunction with the enumeration of the tangled labelings of inflated rooted forests given in \cite{DP20}, seems to imply that the inflation operation on posets is very compatible with promotion. Thus, it would be a natural next step to study promotion on inflations of non-rooted trees. For example, it would be interesting to enumerate the tangled labelings of inflations of simple posets such as $ N $-posets or $ M $-posets. Doing so might generate new methods for attacking \Cref{conj: (n-1)!}, which may also be refined in the following way (we require $ P $ to be connected because of \Cref{reduction to connected case}):
\begin{conjecture}
	Let $ P $ be a connected $ n $-element poset with $ s $ minimal elements. Then $ P $ has at most $ (n-s)(n-2)! $ tangled labelings.
\end{conjecture}

It is also possible to reframe \Cref{conj: (n-1)!} in the following way: Let $ P $ be a connected, $ n $-element poset, and let $ m_1,\ldots,m_s $ be the minimal elements of $ P $. Let $ c:P\to \mathcal{P}([s]) $ be a coloring of $ P $ given by $ i\in c(x) $ if $ x\geq_P m_i $ (here $ \mathcal{P}([s]) $ denotes the power set of $ [s]=\{1,\ldots,s\} $). The following implies \Cref{conj: (n-1)!}. \begin{conjecture}
	With notation as above, \[\Prob\left(c(L_{n-2}\inverse(1))=\{s\}\,|\,L(m_{s-1})=n\right)\geq\Prob\left(c(L_{n-2}\inverse(1))=\{s\}\,|\,L(m_{s})=n\right).\]
\end{conjecture} If the above holds, we may apply the same argument as in \Cref{(n-1)!} to show that the number of tangled labelings increases when we make $ m_{s-1}\lessdot m_s $. Applying this fact repeatedly would prove \Cref{conj: (n-1)!}, since posets with a unique minimal element have exactly $ (n-1)! $ tangled labelings.

\Cref{quasi tangled enumeration} and \Cref{tangled enumeration} together motivate the following conjectures and question: \begin{conjecture}
	If $ P $ is an inflated rooted tree poset with deflated leaves, then the number of tangled labelings of $ P $ is less than or equal to the number of quasi-tangled labelings of $ P $.
\end{conjecture}
\begin{conjecture}
	Let $ P $ be an $ n $-element poset. Then the number of labelings $ L:P\to[n] $ such that $ L_{n-3}\not\in\call(P) $ is at most $ 3(n-1)! $.
\end{conjecture}
\begin{question}
	What is the maximum number of quasi-tangled labelings a poset can have?
\end{question}

\section*{Acknowledgments}
This research was conducted at the Duluth Summer Mathematics Research Program for Undergraduates at the University of Minnesota Duluth with support from Jane Street Capital, the National Security Agency (grant H98230-22-1-0015), the National Science Foundation (grant DMS-2052036), and Harvard University. My research advisors Noah Kravitz, Colin Defant, and Amanda Burcroff provided generous guidance and feedback during the research process for which I am very grateful. I would like to extend special thanks to Swapnil Garg and Noah Kravitz for their invaluable suggestions during the editing process. I would also like to thank Aleksa Milojevi\'{c} for noticing a mistake in one of my proofs early on in the research process. Finally, I am deeply grateful to Joe Gallian for his support and for giving me the opportunity to participate in his research program.

\bibliographystyle{amsalpha}
\bibliography{bib}
\end{document}